\documentclass[twoside]{article}

\usepackage[accepted]{aistats2023}
\usepackage{hyperref}       
\usepackage{url}            
\usepackage{booktabs}       
\usepackage{amsfonts}       
\usepackage{nicefrac}       
\usepackage{microtype}      
\usepackage{xcolor}         
\usepackage{amsmath}
\usepackage{graphicx}
\usepackage{caption}
\usepackage{subcaption}
\usepackage{amsthm}
\usepackage{amssymb}

\usepackage{enumitem}
\usepackage{multirow, wrapfig}

\newcommand{\argmin}[1]{\underset{#1}{\mathrm{argmin}}}
\newcommand{\minimize}[1]{\underset{#1}{\mathrm{minimize}}}
\newcommand{\lmo}{\mathrm{LMO}}

\newcommand{\diam}{\mathrm{diam}}
\newcommand{\gap}{\mathbf{gap}}

\newcommand{\mD}{\mathcal D}
\newcommand{\mE}{\mathcal E}
\newcommand{\mS}{\mathcal S}

\newcommand{\mb}{\mathbb}
\newcommand{\R}{\mathbb R}

\newcommand{\bs}{\mathbf s}
\newcommand{\bS}{\mathbf S}
\newcommand{\bZ}{\mathbf Z}
\newcommand{\bx}{\mathbf x}
\newcommand{\bd}{\mathbf d}
\newcommand{\bg}{\mathbf g}
\newcommand{\bz}{\mathbf z}
\newcommand{\bv}{\mathbf v}
\newcommand{\by}{\mathbf y}
\newcommand{\bh}{\mathbf h}

\newcommand{\FW}{\textsc{FW}}
\newcommand{\multFW}{\textsc{MultFW}}
\newcommand{\FWflow}{\textsc{FWFlow}}
\newcommand{\AvgFW}{\textsc{AvgFW}}
\newcommand{\AvgflowFW}{\textsc{AvgFWFlow}}

\newcommand{\sign}{\mathbf{sign}}

\newcommand{\conv}{\mathbf{conv}}
\newcommand{\diag}{\mathbf{diag}}

\newtheorem{theorem}{Theorem}[section]

\newtheorem{lemma}[theorem]{Lemma}
\newtheorem{proposition}[theorem]{Proposition}
\newtheorem{corollary}[theorem]{Corollary}

\theoremstyle{definition}

\newcommand{\red}[1]{{\color{red}#1}}

\usepackage[round]{natbib}

\begin{document}

%

%

\twocolumn[

\aistatstitle{Reducing Discretization Error in the  Frank-Wolfe Method}

\aistatsauthor{ Zhaoyue Chen \And Yifan Sun }

\aistatsaddress{ Stony Brook University \And Stony Brook University} ]

\begin{abstract}
The Frank-Wolfe algorithm is a popular method in structurally constrained machine learning applications, due to its fast per-iteration complexity. However, one major limitation of the method is a slow rate of convergence that is difficult to accelerate due to erratic, zig-zagging step directions, even asymptotically close to the solution. 
We view this as an artifact of discretization; that is to say, the Frank-Wolfe \emph{flow}, which is its trajectory at asymptotically small step sizes, does not zig-zag, and reducing discretization error will go hand-in-hand in  producing a more stabilized method, with better convergence properties.
We propose two improvements: a multistep Frank-Wolfe method that directly applies optimized higher-order discretization schemes; and an LMO-averaging scheme with reduced discretization error, and whose local convergence rate over general convex sets accelerates from a rate of $O(1/k)$ to up to $O(1/k^{3/2})$.

\end{abstract}

\section{INTRODUCTION}
The Frank Wolfe algorithm (\FW) or the conditional gradient algorithm \citep{LEVITIN19661} is a popular method in constrained convex optimization. It was first developed in \citet{frank1956algorithm} for maximizing a concave quadratic programming problem with linear inequality constraints, and later extended in  \citet{dunn1978conditional}  to minimizing more general smooth convex objective function on a bounded convex set.
More recently, \citet{jaggi2013revisiting} analyzed (\FW) over general convex and continuously differentiable objective functions with convex and compact constraint sets,  and illustrates that when a sparse structural property is desired, the per-iteration cost can be much cheaper than computing projections. 
This has spurred a renewed interest of (\FW) to broad applications in machine learning and signal processing \citep{LacosteJulien2013BlockCoordinateFO,Krishnan2015BarrierFF}, recommender systems \citep{Freund2017AnEF}, image and video co-localization \citep{Joulin2014EfficientIA}, etc. 

Specifically, (\FW) solves the constrained optimization 
\begin{equation}
\minimize{x\in \mD} \quad f(x)
\label{eq:main}
\end{equation}
via the repeated iterations
\[
    \begin{array}{rcl}
\bs_k &=& \argmin{\bs\in \mD} \; \nabla f(\bx_k)^T\bs, \\
\bx_{k+1} &=& \bx_k + \gamma_k (\bs_k - \bx_k).
    \end{array}
    \tag{\FW}
\]
The first operation is often referred to as the \emph{linear minimization oracle (LMO)}, and is the support function of $\mD$ at $-\nabla f(\bx)$:
\[
\lmo_\mD(\bx) := \argmin{\bs\in \mD}\;\nabla f(\bx)^T\bs. 
\]
The iterates $\bs_k$ are usually the vertices of $\mD$ and are often called \emph{atoms}, as their convex combinations build $\bx_k\in \mD$.

In particular, computing the LMO over $\mD$ is often computationally cheap compared to the corresponding projection, especially when $\mD$ is the level set of a sparsifying norm, e.g. the 1-norm or the nuclear norm. 
However, the tradeoff of the cheap per-iteration rate is that the overall convergence rate, in terms of number of iterations $k$, is often much slower than that of projected gradient descent \citep{lacoste2015global,freund2016new}, and in particular shows no signs of acceleration even when $f$ is $\mu$-strongly convex or when momentum-based acceleration techniques are employed. While various acceleration schemes \citep{lacoste2015global} have been proposed and several improved rates given under specific problem geometry \citep{garber2015faster}, by and large the ``vanilla" Frank-Wolfe method, using the ``well-studied step size" $\gamma_{k} = O(1/k)$, can only be shown to reach $O(1/k)$ convergence rate in terms of objective value decrease \citep{Canon1968ATU,jaggi2013revisiting,freund2016new}


 \paragraph{The Zig-Zagging phenomenon.}
The slowness of the Frank-Wolfe method is often explained as a consequence of a potential  ``zig-zagging" phenomenon.
In particular, when the true solution lies on a low dimensional facet and the incoming iterate is angled in a particular way, the method will alternate picking up vertices of this facet, causing a ``zig-zagging" pattern; e.g. $\bx_{k+1}-\bx_{k}$ does not in general point toward $\bx^*$. For this reason, techniques like line search and momentum-based acceleration have little effect.
In fact, methods like the Away-Step Frank Wolfe \citep{lacoste2015global} are designed to counter exactly this, by forcing the iterate to change its angle and approach more directly.

\paragraph{Frank Wolfe flow.} In this paper, we investigate the \emph{dynamical systems interpretation} of the Frank-Wolve method. In particular, we study (\FWflow), whose Euler discretization is (\FW):

\[
\begin{array}{rcl}
s(t)& =& \lmo_\mD(x(t)) \\[1ex]
\dot x(t) &=& \gamma(t) (s(t)-x(t))
    \end{array}
    \;\;(\FWflow)
\]

In particular, (\FWflow)~is an example of Krasovskii regularization \citep{sanfelice2008generalized,krasovskii1968regularization} and thus existance of solutions is ensured.
This dynamical system first studied in \citet{jacimovic1999continuous}, and is a part of the construct presented in \citet{diakonikolas2019approximate}. 
However, neither paper considered the effect of using advanced discretization schemes to better imitate the flow, as a way of improving the method. 
 From an optimization point of view, (\FWflow) is important because, under the family of decay sequences $\gamma(t) = O(1/t)$, its convergence rate is \emph{arbitrarily close to linear}, under the right parameter choices; in contrast, (\FW) is upper bounded by a sublinear $O(1/k)$ convergence rate. 

Viewing the excess error in (\FW) compared to (\FWflow) as discretization error, this work looks at characterizing and attacking this source of error, in efforts of proposing an improved Frank-Wolfe method.  In particular, we argue that this discretization error is particularly detrimental in almost all useful cases of \FW~applied to \emph{sparse optimization applications}, e.g. when $\mD$ is the level-set of the one-norm  or another sparsifying penalty. In this case, improvements found in related works may not apply. 
\begin{itemize}[leftmargin=*]
    \item When the solution $\bx^*$ to \eqref{eq:main} is in the interior of $\mD$, then using \emph{line search}, (\FW) converges with a linear rate if $f$ is strongly-convex \citep{guelat1986some}.   However, this corresponds to a fully dense $\bx^*$, which is not desired in sparse optimization.
    
    \item When the set $\mD$ is strongly convex and $\bx^*$ is on the boundary of $\mD$, or when $\bx^*$ is on a minimal facet of $\mD$, then (\FW) converges with an $O(1/k^2)$ rate \citep{garber2015faster}. However, the 1-norm ball (and other popular choices of $\mD$) are not strongly convex, and the only time when this scenario will apply is if $\bx^*$ has exactly one nonzero.
\end{itemize} 
For these reasons, this paper investigates the more general case where $\bx^*$ is on a \emph{non-minimal facet or interior} of a \emph{non-strongly convex (but convex)} set $\mD$. 

\paragraph{Contributions.} In this regime, we offer three main theoretical contributions.
\begin{itemize}[leftmargin=*]
    \item The \emph{continuous-time Frank-Wolfe method}, e.g. the dynamical system of whose explicit Euler discretization gives (\FW), has a fundamentally faster convergence rate, which is arbitrarily close to linear convergence rate. This suggests that the fundamental bottleneck in the convergence speed is \emph{slowly decaying discretization error}.
    \item While multistep methods reduce discretization error and improve the quality of step directions, overall no finite-window averaging method can fundamentally improve the $O(1/k)$ convergence rate.
    \item Finally, an infinite-window averaging method is introduced, which aggressively attacks the discretization error, and improves the local convergence rate of (\FW) to up to $O(1/k^{3/2})$.
\end{itemize}
The differentiation between local and global convergence  is characterized by the identification of the sparse manifold ($\bar k$, where the LMOs of $\bx_k$ will always be contained in the potential LMOs of $\bx^*$ for all $k \geq \bar k$); this follows the local convergence analysis of \citet{liang2014local, poon2018local, liang2017local,nutini2022let,nutini2019active,hare2004identifying,sun2019we} over general optimization methods.
Overall, these results suggest improved behavior for sparse optimization applications, which is the primary beneficiary of (\FW).

\subsection{Related works}

\paragraph{Continuous-time optimization}
Recent years have witnessed a surge of research papers connecting  dynamical systems with optimization algorithms, generating more intuitive analyses and proposing accelerations. For example, in \citet{su2016differential}, the Nesterov accelerated gradient descent and Polyak Heavy Ball schemes are shown to be discretizations of a certain second-order ordinary differential equation (ODE), whose tunable vanishing friction pertains to specific parameter choices in the methods.

\paragraph{Multistep discretization methods}
Inspired by this analysis, several papers \citep{jingzhao2018direct,shi2019acceleration} have proposed improvements using advanced discretization schemes; \cite{jingzhao2018direct} uses Runge-Kutta integration methods to improve  accelerated gradient methods, and \cite{shi2019acceleration} shows a generalized Leapfrog acceleration scheme which uses a semi-implicit scheme to achieve a very high resolution approximation of the ODE.

Note that for gradient flow, which has the same order of convergence rates as the gradient method, however, it is not likely that multistep methods can offer rate \emph{improvements.} Thus, a unique angle of this work is that in the case of Frank-Wolfe, we will show that indeed, (\FWflow) has a fundamentally faster convergence rate than (\FW). However, with regards to higher order discretization approaches, this work offers a \emph{negative result}, in that no simple explicit higher order discretization scheme fundamentally improves the method's overall convergence rate. In a way, this is a cautionary tale, suggesting that even when discretization error accounts for fundamentally slowness, simple discretization improvements are usually not sufficient.


\paragraph{Accelerated FW.}
A notable work that highlights the notorious zig-zagging phenomenon is
\citet{lacoste2015global}, where an Away-FW method is proposed that cleverly removes offending atoms and improves the search direction. Using this technique, the method is shown to achieve linear convergence under strong convexity of the objective. The tradeoff, however, is that this method requires keeping  past atoms, which may incur an undesired memory cost.
A work that is particularly complementary to ours is \citet{garber2015faster}, which show an improved $O(1/k^2)$ rate when the \emph{constraint set} is strongly convex--this reduces zigzagging since solutions cannot lie in low-dimensional ``flat facets''.
Our work addresses the exact opposite regime, where we take advantage of ``flat facets'' in sparsifying sets (1-norm ball, simplex, etc). This allows the notion of \emph{manifold identification} as determining when suddenly the method behavior improves.

\paragraph{Averaged FW.} Several previous works have investigated \emph{gradient averaging} \citep{acceleratingFW,abernethy2017frank}. While performance seems promising, the rate was not improved past $O(\frac{1}{k})$. 
\citet{ding2020k} investigates \emph{oracle averaging} by solving small subproblems at each iteration to achieve optimal weights.

 
\paragraph{Sparse optimization.}
Other works that investigate \emph{local convergence behavior} include  \citet{liang2014local,liang2017local,poon2018local,sun2019we,iutzeler2020nonsmoothness,nutini2019active, hare2004identifying}.
Here, problems which have these two-stage regimes are described as having \emph{partial smoothness}, which allows for the low-dimensional solution manifold to have significance. 
In our work, 
we differentiate a local convergence regime of when this manifold is ``identified", e.g. all future LMOs are drawn from vertices of this specific manifold. After this point, we show that convergence of both the proposed flow and method can be improved with a faster decaying discretization error, which in practice may be fast. 

\section{THE PROBLEM WITH FW}
\subsection{Small examples}
There are several different explanations as to why the convergence rate of (\FW) appears fundamentally slow, with the predominant theories focused on the bad (zig-zagging) step directions.
Therefore we begin by evaluating the vanilla (\FW) with simple acceleration schemes: line search (LS), Nesterov's extrapolation applied to $\bx_{k}$ direction (Nest. Acc.), and Nesterov's 3-point acceleration as described in \cite{li2020does} (Li Acc). We do this over 5 cases of $\mD$ and $\bx^*$, where $f$ is a strongly convex quadratic.
\begin{itemize}[leftmargin=*]
    \item \textbf{Case 1:} $\mD$ is a convex polytope, and $\bx^*$ is on a minimal facet of $\mD$; in other words, $\bx^*$ is 1-sparse and contains a mixture of exactly 1 atom. In this surprising case, vanilla methods are at least $O(1/k^2)$, Nesterov's acceleration seems to achieve additional acceleration, and line search works instantaneously. We emphasize that this is a deliberately trivial case.
    
    \item \textbf{Case 2:} $\mD$ is a convex polytope, and $\bx^*$ is on a non-minimal facet of $\mD$; here, $\bx^*$  contains a mixture of atoms, more than 1 but less than $n$. This is the typical ``hard" case that we wish to investigate, where no simple acceleration techniques seem to offer any improvement.
    
    \item \textbf{Case 3:} $\mD$ is a convex polytope, and $\bx^*$ is in the relative interior of  $\mD$;  $\bx^*$  contains a full nontrivial mixture of \emph{all} atoms in $\mD$. This case is covered by \cite{guelat1986some}, which suggest linear convergence is possible, \emph{but only when line search is used}. In the absence of line search, all methods \emph{still} do poorly.
    
    \item \textbf{Case 4:}  $\bx^*$ is on the boundary of $\mD$, which is a \emph{strongly convex set.} That is, there exists $\epsilon > 0$ where, for all $x$ and $y$ in $\mathcal D$, any $0\leq\theta \leq 1$, any $z$ such that $\|z\|_2=1$, 
\[
\theta x + (1-\theta) y + \theta(1-\theta) \frac{\epsilon}{2}\|x-y\|^2 z \in \mathcal D.
\]
\citep{vial1982strong}. 
Here, as \cite{garber2015faster} has shown, the presence of a strongly convex $\mD$ allows the iterates to avoid zig-zagging, and in fact even the poorest methods achive $O(1/k^2)$ without any special tricks, with further acceleration possible via Nesterov's extrapolation or line search.
    
    \item \textbf{Case 5:} $\mD$ is a strongly convex set, and  $\bx^*$ is in the relative interior of $\mD$. In fact, this is exactly the same as Case 3: note that the acceleration due to strongly convex $\mD$ is also lost in this case, as the atoms acquired during the local convergence phase do not converge to a small neighborhood. 
\end{itemize}
Figure \ref{fig:casestudy_vanilla} summarizes these observations visually; from the performance plots, it is clear that the open areas (Case 2, and  Cases 3 and 5 without line search) indeed has limited slow convergence.

\onecolumn
\begin{figure}[p]
    \centering
    \includegraphics[width=\linewidth,trim={1cm 1cm 1cm 1cm},clip]{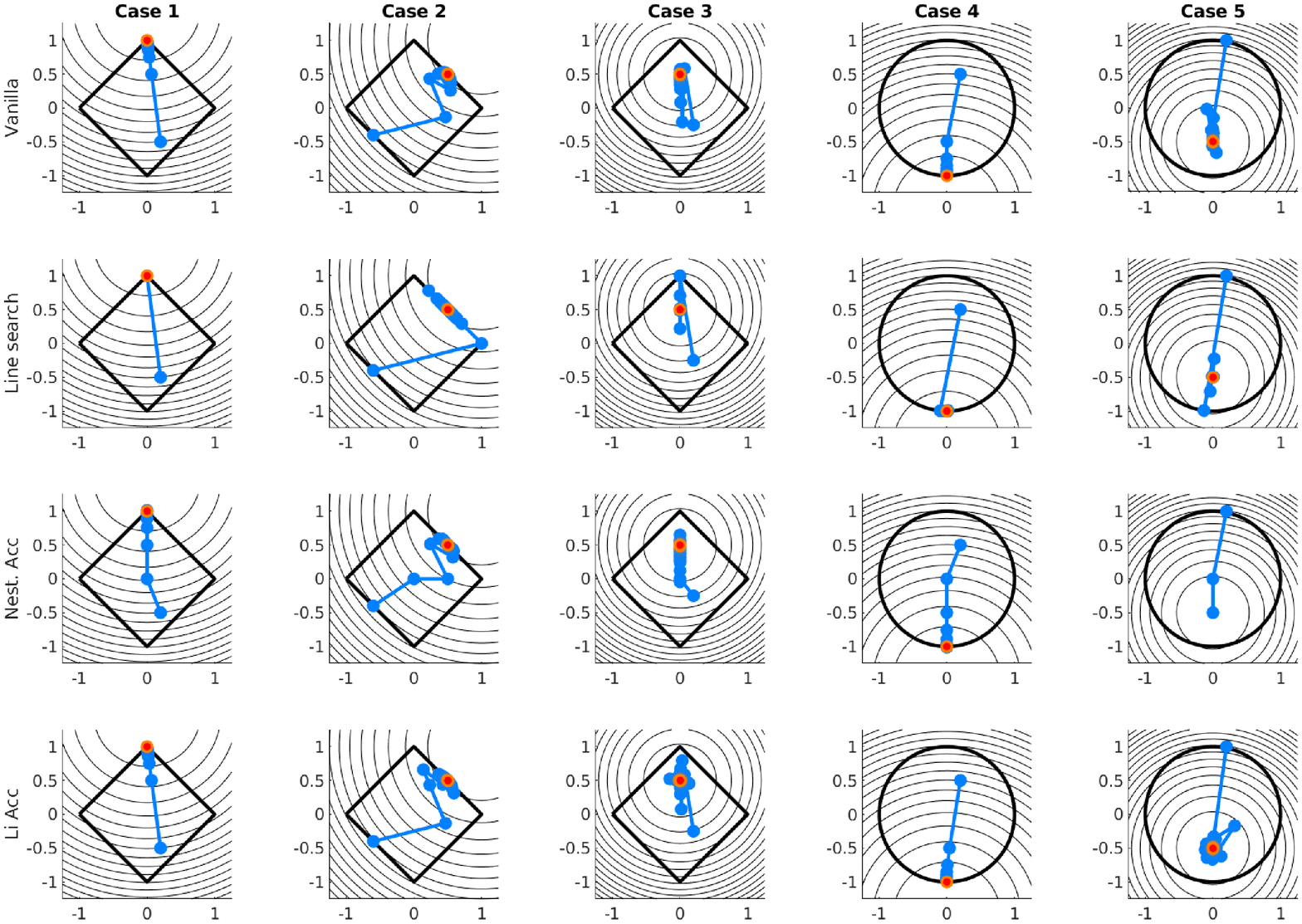}\\
    \includegraphics[width=\linewidth]{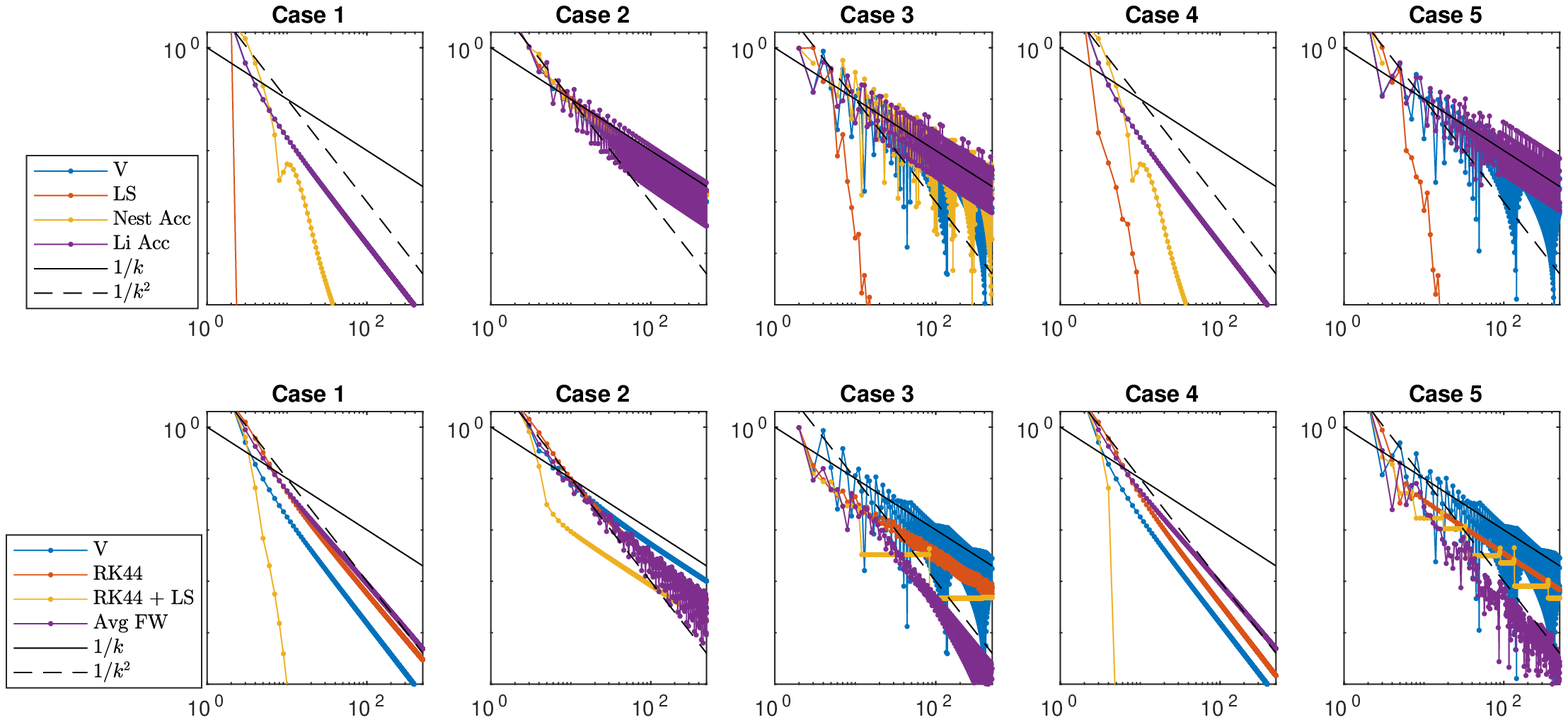} \caption{Small case study of 2-D quadratic problem, under different choices of $\mD$. The last 2 rows gives summarizing gap convergence rates (as a function of iterations $k$) over usual acceleration approaches (upper) and our proposed methods (lower). }
    \label{fig:casestudy_vanilla}
\end{figure}
\twocolumn

\subsection{Two different convergence rates}
We now investigate these problematic behaviors in terms of the (\FWflow) convergence behavior, as compared to that of (\FW). 
First, we derive the convergence rate of (\FWflow). We use the objective value suboptimality as the Lyapunov function:
\[
\mE(t) = f(x(t))-f^*, \quad \mE_k = f(\bx_k)-f^*
\]
where $f^*=f(x^*)$ and $x^*$ minimizes \eqref{eq:main}. Then, using the properties of the LMO and convexity of $f$, 
\begin{eqnarray}
\dot \mE(t) &=& \nabla f(x)^T\dot x(t) \nonumber\\
&=& \gamma(t)\underbrace{\nabla f(x)^T(s-x)}_{\text{negative duality gap}}\nonumber\\
&\leq&-\gamma(t) \mE(t)\label{eq:diffeq-vanilla}
\end{eqnarray}
 and taking $\gamma(t) = \tfrac{c}{c+t}$, we have the flow convergence rate
\[
\frac{\mE(t)}{\mE(0)} \leq \frac{ c^c }{(c+t)^c} \overset{c\to +\infty}{=} \exp(-t).
\]
Note that this rate is \emph{arbitrarily close to a linear rate}.
In contrast, using $L$-smoothness, the (\FW)~method 
satisfies the recursion
\begin{multline*}
f(\bx_{k+1})-f(\bx_k) \leq \\ \gamma_k\nabla f(\bx_k)^T(\bx_{k}-\bs_k) + \underbrace{\frac{L\gamma_k^2}{2}\overbrace{\|\bx_{k}-\bs_k\|_2^2}^{= (2D)^2}}_{\text{discretization term}}\\
\end{multline*}
where $D = \max_{x\in \mD}\|x\|_2$. Note that following this line of reasoning, we can at best bound the difference in $\mE_k$ as
\begin{equation}
\mE_{k+1}-\mE_k \leq -\gamma_k\mE_k + 2LD\gamma_k^2\label{eq:diffeq-vanilla-flow}
\end{equation}
which recursively gives a bound of $\mE_k = O(\tfrac{c}{c+k}) = O(1/k)$. Importantly, this analysis is \emph{tight in general.} 

Note the key difference in \eqref{eq:diffeq-vanilla} and its analogous terms in \eqref{eq:diffeq-vanilla-flow} is the extra $2LD^2\gamma_k^2 = O(1/k^2)$ term, which  throttles the recursion from doing better than $\mE_k = O(1/k)$.
One may ask if it is possible to bypass this problem by simply picking $\gamma_k$ decaying more aggressively; however, then such a sequence becomes summable, and  then  convergence of $\bx_k\to \bx^*$ will not be assured. Therefore, we must have a sequence $\gamma_k$ converging \emph{at least} as slowly as $O(1/k)$.
Thus, the primary culprit in this convergence rate tragedy is the bound $\|\bs_k-\bx_k\|_2=O(D)$ (nondecaying), which forces the discretization term to decay no faster than $O(1/k)$. As shown in our case studies, this is not a loose bound in general. 

In this paper, we investigate methods that push the (\FW) method more toward its (\FWflow) trajectory, using some form of averaging. First, we use the specific weight updates offered by Runge-Kutta multistep schemes, which reduce discretization errors by constant factors, and additionally seem to  improve step direction quality. Second, we propose an averaged LMO method, which gives an overall improved local convergence rate of up to $O(1/k^{3/2})$.

\section{RUNGE-KUTTA MULTISTEP METHOD}
\label{sec:rkmethod}
\subsection{The generalized Runge-Kutta family}
We now look into multistep methods that better imitate the continuous flow by reducing discretization error. 
Observe that the standard (\FW) algorithm is equivalent to the discretization of (\FWflow)~by the explicit Euler's method with step size $\Delta = 1$. It is well known that the discretization error associated with this scheme is $O(\Delta^q)$ with $q = 1$, e.g. it is a method of order 1.

We now consider Runge-Kutta (RK) methods, a generalized class of higher order methods ($q \geq 1$).
Many commonly-used discretization schemes are of the Runge-Kutta family; for example, Dormand–Prince (RK45) is the default ode solver used by MATLAB. Other examples of RK-methods include the    Fehlberg (RK23) and Cash–Karp methods.
These methods are fully parametrized by some choice of $A\in \R^{q\times q}$, $\beta\in \R^q$, and $\omega\in \R^q$ and at step $k$ can be expressed as (for $i = 1,...,q$)
\begin{equation}
\begin{array}{lcl}
\xi_i &=& \displaystyle\dot x\big((k+\omega_i)\Delta ,\;  \bx_{k}+  \sum_{j=1}^q A_{ij} \xi_j\big),\\
\bx_{k+1} &=& \bx_{k}+ \sum_{i=1}^q \beta_i \xi_i.
\end{array}
\label{eq:general_discrete}
  \end{equation}
  For consistency,  $\sum_i \beta_i = 1$, and to maintain explicit implementations,  $A$ is always strictly lower triangular. As a starting point, $\omega_1 = 0$.
  Referring to the description of $\dot x(t)$ as given in (\FWflow), then given a matrix $A\in \R^{q\times q}$ and vectors $\beta,\omega\in \R^q$, the sequence described in \eqref{eq:general_discrete} describes the \emph{q-stage multistep Frank-Wolfe (\multFW) method}. 
  (See also Appendix \ref{app:sec:experiments}.)
  Figure \ref{fig:trajectories} compares pictorially the trajectory of the vanilla (\FW) method with (\multFW), which after averaging has a far more controlled and less erratic trajectory. We hope to leverage this into better convergence behavior.

   
  \begin{proposition}[Feasibility]
  \label{prop:RK-feas}
For a given $q$-stage multistep method defined by $A$, $\beta$, and $\omega$, for each given $k\geq 1$, define
\[
 \bar \gamma_i^{(k)} = \frac{c}{c+k+\omega_i}, \qquad 
 \Gamma^{(k)} = \diag({\bar \gamma^{(k)}}_i),
 \]
 \[
 \mathbf P^{(k)} = \Gamma^{(k)} (I+A^T\Gamma^{(k)})^{-1}, \qquad \mathbf z^{(k)} = q \mathbf P^{(k)}\beta.
\]
Then if $0\leq \mathbf z^{(k)}\leq 1$ for all $k\geq 1$, then 
\[
\bx_{0}\in \mD\Rightarrow \bx_{k}\in \mD, \quad \forall k\geq 1.
\]
  \end{proposition}
 The proof is in  Appendix \ref{app:sec:multistepfw}.
This condition  can be checked explicitly and is true of almost all RK methods with a notable exception of the midpoint method, where $\bz^{(k)}_i<0$ is possible. 
A full table of weights for various RK methods is in Appendix \ref{app:sec:tables}.

\subsection{Convergence of (\multFW)}
We first establish that using a generalized Runge-Kutta method cannot \emph{hurt} convergence, as compared to the usual Frank-Wolfe method.
\begin{proposition} 
All Runge-Kutta methods converge at worst with rate  $f(\bx_{k})-f(\bx^*)\leq O(1/k)$.
\label{prop:main-rungekutta-positive}
\end{proposition}
The proof is in Appendix \ref{app:sec:rkpositiveresults}. 
That is to say, (\multFW) cannot be an order \emph{slower} than vanilla (\FW).
\footnote{It should be noted, however, that the convergence rate in terms of $k$ does not account for the extra factor of $q$ gradient calls needed for a $q$-stage method. While this may be burdensome, it does not increase the \emph{order} of convergence rate.}

\begin{figure*}
    \centering
    \includegraphics[width=5.5in,trim={2.5cm 8cm 2.5cm 0},clip]{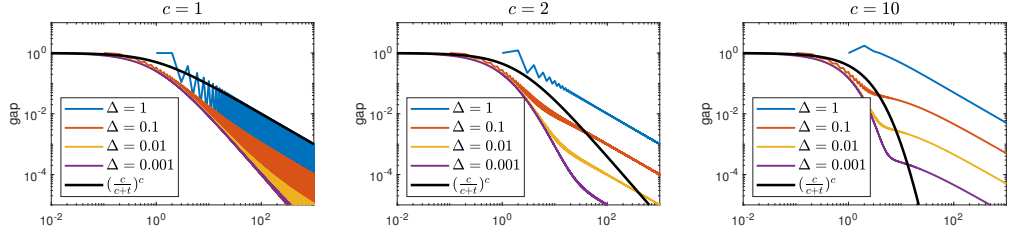}
    \caption{\textbf{Continuous vs discrete.} A comparison of the numerical error vs 
    compared with derived rate. The black curve shows the upper bound on the flow rate, compared against simulated method rates for smaller discretization units. The two-stage behavior of the curves is intriguing, as it seems there is a fundamental point where discretization error takes over, and forces the $O(1/k)$ rate to manifest.}
    \label{fig:limit_continuous}
\end{figure*}
It is left to ask if (\multFW) can \emph{improve} upon (\FW).
Recall that (\FWflow) achieves a rate of $O(1/t^c)$ rate, and taking $c\to +\infty$ achieves a linear rate.
This is also verified numerically in Figure \ref{fig:limit_continuous}; larger $c$ provides a sharper  convergence rate. 
It is hence tempting to think that increasing $c$ can help (\FW) methods in general, and in particular by adapting a higher order multistep method, we can overcome the problems caused by discretization errors. 


\paragraph{Lower bound derivation: a toy problem.}
Let us now consider a simple bounded optimization problem over scalar variables $\bx \in \mathbb R$:
\begin{equation}
        \min_{\bx} \; f(\bx)\quad \mathrm{s.~t.} \; -1\leq \bx \leq 1
        \label{eq:toyproblem}
\end{equation}
where 
\[
f(\bx) = 
\begin{cases}
\bx^2/2 & \text{if } |\bx| < \varepsilon\\
\varepsilon \bx-\varepsilon^2/2 & \text{if } \bx \geq \varepsilon\\
-\varepsilon \bx - \varepsilon^2/2 & \text{if } \bx \leq -\varepsilon\\
\end{cases}
\]
which is a scaled version of the Huber norm applied to scalars. By design, no matter how small $\varepsilon$ is, $f$ is $L$-smooth and 1-Lipschitz. Additionally,  $\lmo_{[-1,1]}(\bx) = -\sign(\bx)$. 

\begin{proposition} 
Assuming that $0 <q\mb P^{(k)} \beta < 1$ for all $k$.
Start with $\bx_{0} = 1$.
Suppose 
the choice of $\beta\in \R^p$ is not ``cancellable"; that is, there exist no partition $S_1\cup S_2 =\{1,...,p\}$ where 
\[
\sum_{i\in S_1}\beta_i - \sum_{j\in S_2} \beta_j = 0.
\]
Then regardless of the order $p$ and choice of $A$, $\beta$, and $\omega$, as long as $\sum_i \beta_i = 1$, then 
\[
\sup_{k'>k} |\bx_{k}| =\Omega(1/k).
\]
That is, the tightest upper bound is $O(1/k)$.
\label{prop:main-rungekutta-negative}
\end{proposition}

The proof is in Appendix \ref{app:sec:negativeresults}. 
The assumption of a ``non-cancellable" choice of $\beta_i$ may seem strange, but in fact it is true for most of the higher order  Runge-Kutta methods. More importantly, the assumption doesn't matter in practice; even if we force $\beta_i$'s to be all equal, our numerical experiments do not show much performance difference in this toy problem.  (Translation: do not design your Runge Kutta method for the $\beta$'s to be cancel-able in hopes of achieving a better rate!)

Proposition \ref{prop:main-rungekutta-negative} implies a $\Omega(1/k)$ bound on $|\bx_k|$. To extend it to an $\Omega(1/k)$ bound on $f(\bx_k)-f^*$, note that whenever $|\bx| \geq \varepsilon$,
\[
\frac{f(\bx_k)-f^*}{f(\bx_0)-f^*} \geq \frac{|\bx_k|}{2|\bx_0|}
\]
for $\varepsilon$ \emph{arbitrarily small}. 
\begin{corollary}
Although higher order RK methods may provide constant factor improvements, the worst best case bound for (\multFW), for any RK method, is of order $O(1/k)$.
\end{corollary}

Although this negative result is disappointing, it is valuable to know that simple multistep enhancements of traditional optimization methods may not single-handily produce miraculously better convergence rates.


\section{AVERAGED FRANK-WOLFE METHOD}
We now propose an infinite-window LMO-averaging Frank-Wolfe (AvgFW) method, by replacing $\bs_k$ with an averaged version $\bar \bs_k$. 
\[
\begin{array}{rcl}
\bs_k &=& \lmo_\mD(\bx_k)\\[1ex]
\bar \bs_{k} &=& \bar \bs_{k-1} + \beta_k (\bs_k-\bar \bs_{k-1})\\[1ex]
 \bx_{k+1} &=&  \bx_k + \gamma_k (\bar \bs_k - \bx_k)
\end{array}
\tag{\AvgFW}
\]

where $\beta_k = (\frac{c}{c+k})^p$ for $c \geq 0$ and $0< p \leq 1$. \footnote{Recall that $\gamma_k = c/(c+k)$. While it is possible to use $\beta_k = b/(b+k)$ with $b\neq c$ in practice, our proofs considerably simplify when the constants are the same.}
Here, the smoothing in $\bar \bs_k$ has two roles.
First, averaging reduces zigzagging; at every $k$, $\bar \bs_k$ is a convex combination
 of past $\bs_k$, and has a smoothing effect that qualitatively also reduces zig-zagging; this is of importance should the user wish to use line search or momentum-based acceleration.
 Second, this forces the discretization  term $\|\bx_{k}-\bar \bs_k\|_2$ to \emph{decay}, allowing for faster numerical performance. 
 Figure \ref{fig:trajectories} shows a simple 2D example of our (\AvgFW) method (right) compared to the usual (\FW) method (left) and the Runge-Kutta multistep version (\multFW) (center). Note that (\AvgflowFW) is fundamentally different than (\FWflow), and is better mirrored by (\AvgFW).
 
 \begin{figure}
\centering
\includegraphics[width=\linewidth,trim={1cm 0 1cm 0},clip]{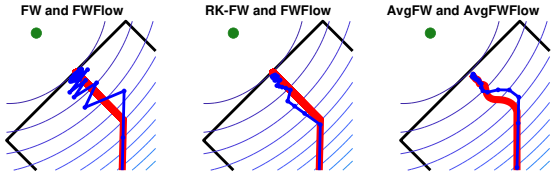}
\caption{Panels showing trajectory behavior of (\FW) (left), (\multFW) (center), and (\AvgFW) (right). The method trajectories (blue) and underlying flow (red) are both shown.}
\label{fig:trajectories}
\end{figure}

Using similar tricks as in \cite{jacimovic1999continuous}, we may view (\AvgFW) as an Euler discretization (with $\Delta = 1$) of the following dynamical system
\[
\begin{array}{rcl}
s(t) &=& \lmo_\mD(x(t))\\[1ex]
{\dot {\bar s}(t)}&=&  \beta(t) (s(t)-\bar s(t))\\[1ex]
\dot x(t) &=&   \gamma(t) (\bar s(t)-x(t))
\end{array}
\tag{\AvgflowFW},
\]
where $\gamma(t)$, $\beta(t)$ are such that $\gamma_k = \gamma(k\Delta)$ and $\beta_k = \beta(k\Delta)$.

\subsection{Global convergence}
We start by showing the method converges, with no assumptions on sparsity or manifold identification. We note that in practice we see faster convergence of (\AvgFW) compared to (\FW), despite weaker theoretical rates.


\begin{theorem}[Global rates]  Take $0< p < 1$, and assume $f$ is $L$-smooth and $\mu$-strongly convex.
\begin{itemize}\itemsep 0pt
    \item For $\beta(t) = \left(\frac{c}{c+t}\right)^p$, the flow (\AvgflowFW) satisfies
$f(x(t)) \leq O\left(\frac{1}{t^{1-p}}\right).$
\item For $\beta_k = \left(\frac{c}{c+k}\right)^p$, the method (\AvgFW) satisfies
$f(\bx_{k}) \leq O\left(\frac{1}{k^p}\right).$
\end{itemize}

\end{theorem}
The proofs are in Appendix \ref{app:sec:averagefw}. Note that the rates are not exactly reciprocal. In terms of analysis, the method is that  allows the term $\beta_k\gap(\bx_k)$ to take the weight of an entire step, whereas in the flow, the corresponding term is infinitesimally small. This is a curious disadvantage of the flow analysis; since most of our progress is accumulated in the last step, the method exploits this more readily than the flow, where the step is infinitesimally small.

\subsection{Manifold identification}
Now that we know that the method converges, we may discuss the local convergence regime, characterized by manifold identification. Specifically, 
the \emph{manifold is identified} at $\bar k$ if any $k > \bar k$, any $\lmo_\mD(\bx_k)$ is also an $\lmo_\mD(\bx^*)$. 
For example, when $\mD$ is the one-norm ball, then the manifold is identified when for all $k\geq \bar k$, $\bs_k$ are 1-hot vectors with nonzero supports contained within the sparsity of $\bx^*$. This is \emph{guaranteed to happen for some finite $\bar k$} under mild degeneracy assumptions, and is the basis for gap-safe screening procedures \citep{ndiaye2017gap,sun2020safe}.

\subsection{Accelerated local convergence}

\begin{theorem}[Local rate]
Assume  $f$ is $L$-smooth and $\mu$-strongly convex. After manifold identification, the flow (\AvgflowFW) satisfies
\[
f(x(t))-f(x^*) \leq \frac{\log(t)}{t^c}.
\]
\end{theorem}
The proof is in Appendix \ref{app:sec:local}.


\begin{theorem}[Local rate]
Assume that $f$ is $\mu$-strongly convex, and pick $c \geq 3p/2+1$. After manifold identification, the method (\AvgFW) satisfies
\[
f(\bx_k)-f(\bx^*) = O(1/k^{3p/2}).
\]
\end{theorem}
The proof is in Appendix \ref{app:sec:global}.
Although the proof requires $p<1$, in practice, we often use $p = 1$ and observe about a $O(1/k^{3/2})$ rate, regardless of strong convexity.

\section{NUMERICAL EXPERIMENTS}


We now compare the Frank-Wolfe method (\FW), its multistep form (\multFW), and its averaged form (\AvgFW), against two baseline methods:
\begin{itemize}
\item the averaged gradient method of \cite{abernethy2017frank},
\item and the away-step method of \cite{lacoste2015global}. 
\end{itemize}
In all cases, when appropriate, we also add line search (LS). Note that while the away-step method is theoretically a superior method to those we propose, it requires some computational overhead in maintaining and readjusting a full history of past atoms, and also always requires  line search to ensure feasiblity. Since \multFW~requires compute additional gradients, our  $x$-axis gives the number of gradient calls, instead of iterations, to ensure  fair  comparison.

\paragraph{Simulated compressive sensing.}
In Figure \ref{fig:compressed_sensing} we simulate compressive sensing using a $\ell_1$- norm ball constraint. Given $\bx_0\in \R^{m}$, a sparse ground truth vector with 10\% nonzeros, and given $A\in \R^{n\times m}$ with entries i.i.d. Gaussian, we generate $\by = A\bx_0$ (noiseless observation). The problem formulation is
\begin{equation}
\min_{x\in \R^n} \quad  \tfrac{1}{2}\|Ax-y\|_2^2\qquad
\mathrm{subject~to} \quad  \|x\|_1\leq\alpha.
\end{equation}

\begin{figure}[!htb]
  \includegraphics[width=\linewidth,trim={0cm 0 0cm 0},clip]{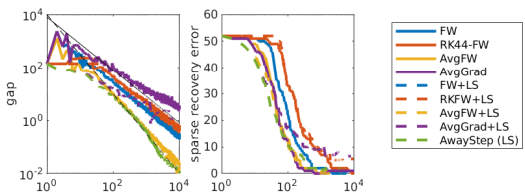}
  \caption{\textbf{Compressed sensing.} $n = m = 500$, $c = \beta = 2$, $p=1$.  }
  \label{fig:compressed_sensing}
\end{figure}
\paragraph{Signed compressive sensing.}
In Fig. \ref{fig:signed_compressive_sensing}, under the same parameter distribution as in the previous example, we now use a sampling model where $y_i = \sign(a_i^T\bx_0)$, and attempt to recover the sparsity of $\bx_0$ using sparse logistic regression
\begin{equation}
    \min_{x\in \R^n} \;  \frac{1}{m}\sum_{i=1}^{m}\log(1+\exp(-y_ia_i^Tx))\quad
    \mathrm{s.t.} \;  \|x\|_1\leq\alpha.
\end{equation}
Because signed compressive sensing is a harder task, we use a higher sampling ratio of 50x.




\begin{figure}[!htb]
  \includegraphics[width=\linewidth,trim={1cm 0 1cm 0},clip]{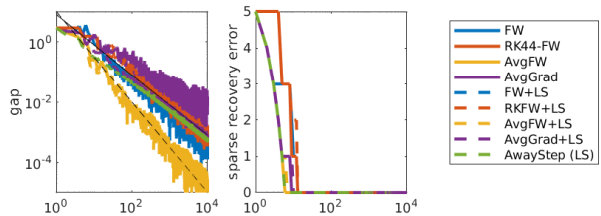}
  \caption{\textbf{Signed compressed sensing.}  $m = 5000$, $n = 100$, $c = \beta = 2$, $p=1$. } 
  \label{fig:signed_compressive_sensing}
\end{figure}


\section{CONCLUSION}
The main goal of this study is to see if the discretization error, which seems to be the primary cause of slow (\FW) convergence, can be attacked directly in order to produce a better method. First, we highlight the artifacts and pitfalls of discretization error (largely, bad zig-zagging step directions that are not amenable to momentum-based acceleration) and the regimes where they are most prominent (on the non-minimal boundary of a polytope, or in the interior of a set when line search is not used).

Next, we explore the use of \emph{higher order discretization methods}, which has been used in the past in gradient flow applications with observed success, and which has fundamentally better truncation errors with each growing order. However, we show that though this method gives constant factor improvements in practice, asymptotically the worst case convergence rate does not improve beyond $O(1/k)$. This is consistent with simulations that show initially fast convergence, but a ``tapering off'' once the residual convergence error is below that of the truncation error order guarantee. While this result is disappointing, it is important to highlight the main flaw of directly adapting this popular dynamical systems tool for optimization improvement. 

Finally, we explore an LMO-averaging method, which produces a smoother convergence trajectory and a fundamentally improved convergence rate (from $O(1/k)$ to up to $O(1/k^{3/2})$) with negligible computation and memory overhead.
This method is largely inspired from viewing the multistep method as a finite-window averaging method, which achieves finite constant order rate improvement--thus, only an infinite window seems to allow for non-constant order improvement.
Our numerical results show that, though our theoretical improvements are only local, the effect of this acceleration appears effective globally; moreover, manifold identification (or at least reduction to a small working set of nonzeros) appears almost immediately in many cases. 

\paragraph{Use in gap-safe screening rules.}
An important related application to these works is gap-safe convergence rates \cite{ndiaye2017gap,sun2020safe}, which produce sparsity guarantees given intermediate iterates $\bx_{k}$, provided the duality gap at $\bx_{k}$ is below some problem-dependent constant. Typically, these guarantees are strong, but slow to realize in practice, especially when gap convergence is slow. In this, (\AvgFW) offers a distinct advantage in its faster gap convergence, in providing this guarantee. 


\newpage
\bibliography{main.bbl}

\appendix
\onecolumn

\onecolumn
\section{ADDITIONAL EXPERIMENTS}

We include a few more experiments to expand beyond Runge-Kutta family (see Figure \ref{figs:additional_multi}). In general, we do not notice much difference in \multFW~with different multistep methods. It is consistent with our theory, which states that \multFW~is still $\Omega(1/k)$.

\begin{figure}[!htb]
\centering
  \includegraphics[width=0.4\linewidth,trim={0cm 0 0cm 0},clip]{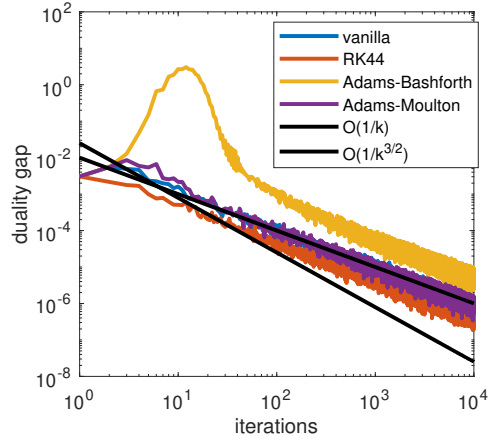}
         \caption{Additional multistep methods}
         \label{figs:additional_multi}
\end{figure}

\section{EXTRA TABLEAUS OF HIGHER ORDER RUNGE KUTTA METHODS}
\label{app:sec:tables}

\begin{itemize}
    \item Midpoint method
    \[
    A = \begin{bmatrix}
    0 & 0 \\ 1/2 & 0
    \end{bmatrix},
    \qquad
    \beta = \begin{bmatrix}
    0 \\1 
    \end{bmatrix},
    \qquad
    \omega = \begin{bmatrix}
    0\\1/2
    \end{bmatrix},\qquad
    \bz^{(1)} \approx \begin{bmatrix}
    -0.3810\\1.1429
    \end{bmatrix},\qquad
    \bz^{(2)} \approx \begin{bmatrix}
    -0.2222\\0.8889
    \end{bmatrix}
\]
     \item Runge Kutta 4th Order Tableau (44)
    \[
    A = \begin{bmatrix}
   0    & 0    & 0 & 0 \\
     1/2  & 0    & 0 & 0 \\
     0  & 1/2  & 0 & 0 \\
    0    & 0    & 1 & 0 
    \end{bmatrix},
    \qquad
    \beta = \begin{bmatrix}
  1/6   \\ 1/3  \\ 1/3 \\ 1/6
    \end{bmatrix},
    \qquad
    \omega = \begin{bmatrix}
    0   \\
        1/2 \\
        1/2 \\
        1  
    \end{bmatrix},\qquad
    \bz^{(1)} \approx \begin{bmatrix}
    0.2449\\
    0.5986\\
    0.5714\\
    0.3333
    \end{bmatrix}
    \]
    
        \item Runge Kutta 3/8 Rule Tableau (4)
    \[
    A = \begin{bmatrix}
   0    & 0    & 0 & 0 \\
     1/3  & 0    & 0 & 0 \\
     -1/3  & 1  & 0 & 0 \\
     1    & -1    & 1 & 0 \\
    \end{bmatrix},
    \qquad
    \beta = \begin{bmatrix}
 1/8   \\3/8  \\ 3/8 \\ 1/8
    \end{bmatrix},
    \qquad
    \omega = \begin{bmatrix}
   0    \\
        1/3  \\
        2/3 \\
        1  
    \end{bmatrix},\qquad
    \bz^{(1)} \approx \begin{bmatrix}
    0.1758\\
    0.6409\\
    0.6818\\
    0.2500
    \end{bmatrix}
    \]
\item 
Runge Kutta 5 Tableau
\[
    A = \begin{bmatrix}
 0 & 0 & 0 & 0 & 0 & 0\\
 1/4 & 0 & 0 & 0 & 0 & 0\\
  1/8 & 1/8 & 0 & 0 & 0 & 0\\
 0 & -1/2 & 1 & 0 & 0 & 0\\
 3/16 & 0 & 0 & 9/16 & 0 & 0\\
 -3/7 & 2/7 & 12/7 & -12/7 & 8/7 & 0\\
    \end{bmatrix},
    \qquad
    \beta = \begin{bmatrix}
7/90 \\ 0 \\ 32/90 \\ 12/90 \\ 32/90 \\7/90
    \end{bmatrix},
    \qquad
    \omega = \begin{bmatrix}
    0 \\
        1/4 \\
        1/4 \\
        1/2 \\
        3/4\\
        1 
    \end{bmatrix},\qquad
    \bz^{(1)} \approx \begin{bmatrix}
      0.1821\\
    0.0068\\
    0.8416\\
    0.3657\\
    0.9956\\
    0.2333
    \end{bmatrix}
    \]
\end{itemize}
In all examples, $\|\bz^{(k)}\|_\infty$ monotonically decays with $k$.

\section{EXTRA EXPERIMENTS}
\label{app:sec:experiments}


Figure \ref{fig:continuous_no_zigzag} quantifies this notion more concretely. We  measure ``zig-zagging energy" by averaging the deviation of each iterate's direction across $k$-step directions, for $k = 1,...,W$, for some measurement window $W$:
\[
\mE_{\mathrm{zigzag}}(\bx^{(k+1)},...,\bx^{(k+W)}) = \frac{1}{W-1}\sum_{i=k+1}^{k+W-1} \Big\|\underbrace{\left(I-\frac{1}{\|{\bar\bd}^{(k)}\|_2}\bar\bd^{(k)}({\bar\bd}^{(k)})^T\right)}_{\mb Q}\bd^{(i)}\Big\|_2,
\]
where $\bd^{(i)} = \bx^{(i+1)}-\bx^{(i)}$ is the current iterate direction and $\bar\bd^{(k)} = \bx^{(k+W)}-\bx^{(k)}$ a ``smoothed" direction. The projection operator $\mb Q$ removes the component of the current direction  in the direction of the smoothed direction, and we measure this ``average deviation energy."
We divide the trajectory into these window blocks, and report the average of these measurements $\mE_{\mathrm{zigzag}}$ over $T=100$ time steps (total iteration = $T/\Delta$). 
Figure \ref{fig:continuous_no_zigzag} (top table) exactly shows this behavior, where the problem is sparse constrained logistic regression minimization over several machine learning classification datasets~\citep{guyon2004result} (Sensing (ours), Gisette \footnote{Full dataset available at \url{https://archive.ics.uci.edu/ml/datasets/Gisette}. We use a subsampling, as given in 
\url{https://github.com/cyrillewcombettes/boostfw}.
} and Madelon \footnote{Dataset: \url{https://archive.ics.uci.edu/ml/datasets/madelon}})
are shown in Fig. \ref{fig:continuous_no_zigzag}.

\begin{figure*}
    \centering
    \begin{tabular}{lr}
    \begin{minipage}{.2\textwidth}
    \includegraphics[width=1.5in]{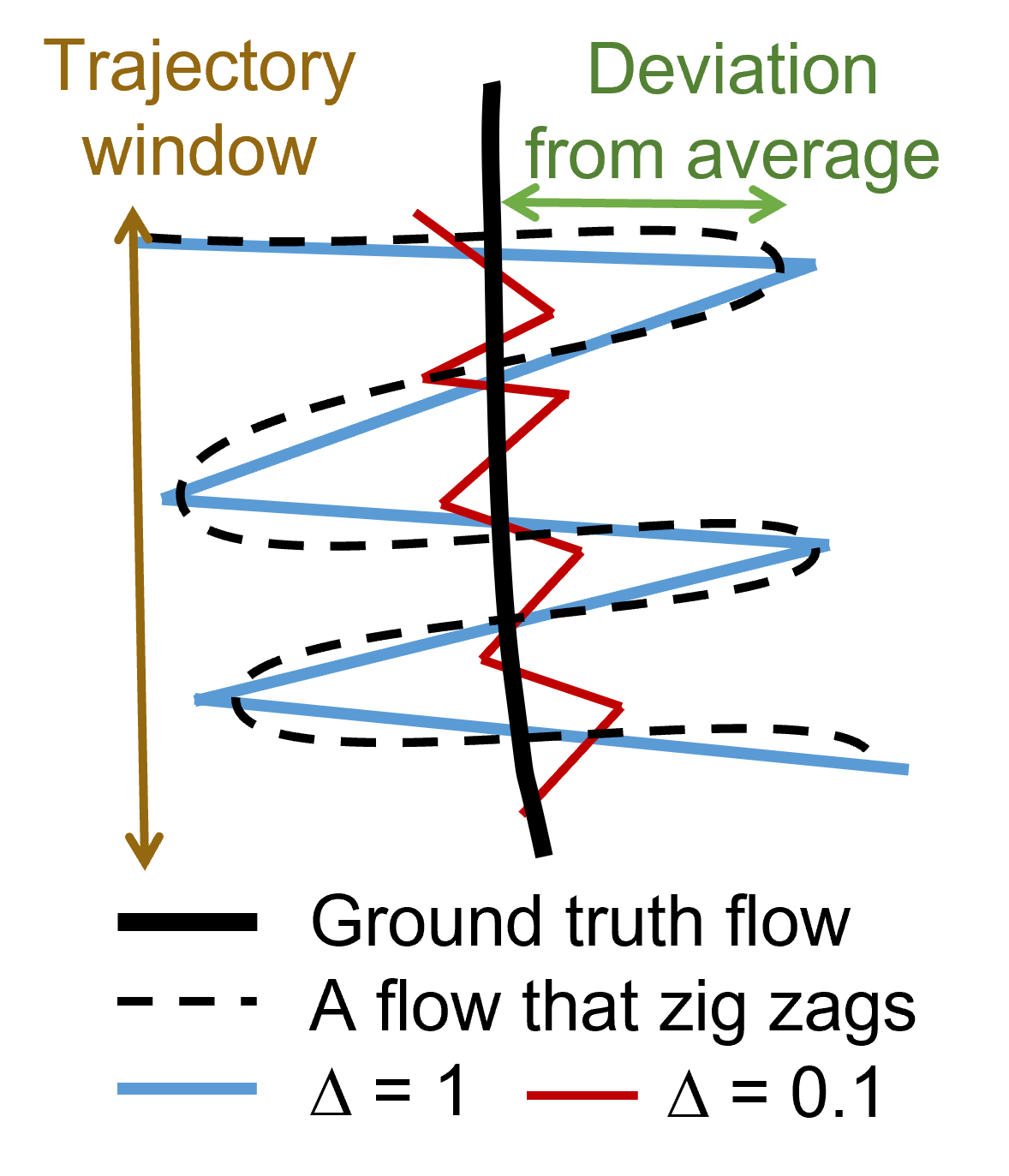}
    \end{minipage}
    &
    \begin{minipage} {.65\textwidth}
    \begin{tabular}{c}
    
    \begin{minipage}{.8\textwidth}
    \begin{tabular}[t]{c|c|c|c}
    \hline
    Test set&$\Delta = 1$&$\Delta = 0.1$ & $\Delta = 0.01$\\
    \hline
    Sensing & 105.90 / 140.29 & 10.49 / 13.86 & 1.05 / 1.39\\
    Madelon &0.11 / 0.23 &0.021 / 0.028&0.0021 / 0.0028\\
    Gisette &1.08 / 1.74&0.21 / 0.28&0.021 / 0.028\\
    \hline
    \end{tabular}
    \begin{center}
        Zigzagging in continuous flow
    \end{center}
    \end{minipage}
    \\
     \\
    \begin{minipage}{.8\textwidth}
    \begin{tabular}[t]{c|c|c|c}
    \hline
    Test set& FW & FW-MID & FW-RK4\\
    \hline
    Sensing & 105.90 / 140.29 & 0.57 / 1.0018 & 0.015 / 0.033\\
    Madelon & 0.031 / 0.040 & 0.025 / 0.029 &0.025 / 0.029\\
    Gisette & 0.30 / 0.40 & 0.25 / 0.11 & 0.22 / 0.20\\
    \hline
    \end{tabular}
    \begin{center}
        Zigzagging in multistep methods
    \end{center}
    \end{minipage}
    \end{tabular}
    
    \end{minipage}
    
    \end{tabular}
    
    \caption{\textbf{Zig-zagging on real datasets.} 
    Average deviation of different discretizations of \FWflow. Top table uses different $\Delta$s and uses the vanilla Euler's discretization (FW). Bottom uses $\Delta = 1$ and different multistep methods.  
    The two numbers in each box correspond to window sizes 5 / 20.
    }
    \label{fig:continuous_no_zigzag}
\end{figure*}




\section{MULTISTEP FRANK WOLFE PROOFS}
\label{app:sec:multistepfw}
\subsection{Feasibility}

\begin{proposition}
  \label{prop:RK-feas}
For a given $q$-stage RK method defined by $A$, $\beta$, and $\omega$, for each given $k\geq 1$, define
\[
 \bar \gamma_i^{(k)} = \frac{c}{c+k+\omega_i}, \qquad 
 \Gamma^{(k)} = \diag({\bar \gamma^{(k)}}_i),
 \]
 \[
 \mathbf P^{(k)} = \Gamma^{(k)} (I+A^T\Gamma^{(k)})^{-1}, \qquad \mathbf z^{(k)} = q \mathbf P^{(k)}\beta.
\]
Then if $0\leq \mathbf z^{(k)}\leq 1$ for all $k\geq 1$, then 
\[
\bx_{0}\in \mD\Rightarrow \bx_{k}\in \mD, \quad \forall k\geq 1.
\]
  \end{proposition}
  
  Proof of Prop. \ref{prop:RK-feas}.
  \begin{proof}
  For a given $k$, construct additionally
  \[
\mathbf Z = \begin{bmatrix} \xi_1 & \xi_2 & \cdots & \xi_q \end{bmatrix},
\]
\[
\bar{\mathbf X} = \begin{bmatrix} \bar \bx_1 & \bar \bx_2 & \cdots & \bar \bx_q \end{bmatrix},\quad
\bar{\mathbf S} = \begin{bmatrix} \bar \bs_1 & \bar \bs_2 & \cdots & \bar \bs_q \end{bmatrix}.
\]
where
      \begin{eqnarray*}
 {\bar \bx}_i &=& \bx_{k}+  \sum_{j=1}^q A_{ij} \xi_j, \\
 {\bar \bs}_i &=& \lmo({\bar \bx}_i).
\end{eqnarray*}
  Then we can rewrite \eqref{eq:general_discrete} as
  \begin{eqnarray*}
  \mathbf Z &=& (\bar {\mathbf S} - \bar{\mathbf X})\Gamma = (\bar {\mathbf S} - \bx_{k}\mathbf 1^T - \mathbf ZA^T)\Gamma \\
  &=& (\bar {\mathbf S} - \bx_{k}\mathbf 1^T )\mathbf P
  \end{eqnarray*}
  for shorthand  $\mathbf P = \mathbf P^{(k)}$ and $\Gamma=\Gamma^{(k)}$. 
 Then
 \begin{eqnarray*}
\bx_{k+1} 
&=& \bx_{k}(1-\mathbf 1^T \mathbf P\beta) + \bar {\mathbf S} \mathbf P\beta\\
&=&\frac{1}{q}\sum_{i=1}^q  \underbrace{(1-\bz^{(k)}_i)\bx_{k} + \bz^{(k)}_i\bar \bs_i }_{\hat\xi_i}
  \end{eqnarray*}
  where $\bz^{(k)}_i$ is the $i$th element of $\bz^{(k)}$, and $\beta = (\beta_1,...,\beta_q)$. 
Then if $0\leq \bz^{(k)}_i\leq 1$, then $\hat \xi_i$ is a convex combination of $\bx_{k}$ and $\bar \bs_i$, and $\hat \xi_i\in \mD$ if $\bx_{k}\in \mD$. Moreover, $\bx_{k+1}$ is an average of $\hat\xi_i$, and thus $\bx_{k+1}\in \mD$. Thus we have recursively shown that $\bx_{k}\in \mD$ for all $k$.
  \end{proof}

\subsection{Positive Runge-Kutta convergence result}
\label{app:sec:rkpositiveresults}

\begin{lemma} 
After one step, the generalized Runge-Kutta method satisfies
\[
h(\bx_{k+1})-h(\bx_{k}) \leq
-\gamma_{k+1} h(\bx_{k}) + D_4(\gamma_{k+1})^2\]

where $h(\bx) = f(\bx) - f(\bx^*)$ and


\[
D_4 =\frac{LD_2^2+2LD_2D_3+2D_3}{2},\quad D_2 = c_1 D, \quad D_3 = c_2c_1 D, \quad c_1 = qp_{\max}, \quad c_2 = q \max_{ij} |A_{ij}|, \quad D = \diam(\mD).
\]

\label{lem:rungekutta-positive-onestep}
\end{lemma}
\begin{proof}
For ease of notation, we write $\bx = \bx_{k}$ and $\bx^+ = \bx_{k+1}$. We will use $\gamma=\gamma_{k} = \tfrac{c}{c+k}$, and $\bar \gamma_i = \tfrac{c}{c+k+\omega_i}$.
Now consider the generalized RK method
\begin{eqnarray*}
\bar \bx_i &=& \bx + \sum_{j=1}^q A_{ij} \xi_j\\
\xi_i &=& \underbrace{\frac{c}{c+k+\omega_i}}_{\tilde \gamma_i} ( \bs_i - \bar {\bx}_i )\\
\bx^+&=&  \bx + \sum_{i=1}^q \beta_i \xi_i\\
\end{eqnarray*}
where $\bs_i = \lmo(\bar \bx_i)$.


Define $D = \diam(\mD)$. 
We use the notation from  section
\ref{sec:rkmethod}. 
Denote the 2,$\infty$-norm as
\[
\|A\|_{2,\infty} = \max_j \|a_j\|_2
\]
where $a_j$ is the $j$th column of $A$. Note that all the element-wise elements in 
\[
\mathbf P^{(k)} = \Gamma^{(k)}(I+A^T\Gamma^{(k)})^{-1}
\]
is a decaying function of $k$, and thus defining $p_{\max} = \|\mathbf P^{(1)}\|_{2,\infty}$
we see that
\[
\|\bar {\mathbf Z}\|_{2,\infty} = \|(\bar {\mathbf S} - \bx^{(k)}\mathbf 1)\mathbf P^{(k)}\|_{2,\infty} \leq qp_{\max} D.
\]

Therefore, since $\bar {\mathbf Z} = (\bar {\mathbf S}-\bar {\mathbf X})\Gamma$, and all the diagonal elements of $\Gamma$ are at most 1, 
\[
\|\bs_i-\bar \bx_i\|_2  \leq qp_{\max} D =: D_2
\]
and
\[
\|\bx-\bar \bx_i\|_2 = \|\sum_{j=1}^q A_{ij} \gamma_j (\bs_j-\bar \bx_j)\|_2 \leq q \max_{ij} |A_{ij}| \gamma D_2 =: D_3 \gamma.
\]

Then 
\begin{eqnarray*}
f(\bx^+)-f(\bx) &\leq&  \nabla f(\bx)^T(\bx^+-\bx) + \frac{L}{2}\|\bx^+-\bx\|_2^2\\
&=&  \sum_i \beta_i \tilde \gamma_i \nabla f(\bx)^T(\bs_i-\bar \bx_i) + \frac{L}{2}\underbrace{\|\sum_i \beta_i \tilde \gamma_i (\bs_i-\bar \bx_i)\|_2^2}_{\leq \gamma^2 D_2^2}\\
&=&  \sum_i \beta_i \tilde \gamma_i (\nabla f(\bx)-\nabla f(\bar \bx_i))^T(\bs_i-\bar \bx_i) +
 \sum_i \beta_i \tilde \gamma_i \underbrace{\nabla f(\bar \bx_i)^T(\bs_i-\bar \bx_i)}_{-\gap(\bar \bx_i)} +
 \frac{L\gamma^2D_2^2}{2}\\
 &\leq& \sum_i \beta_i\underbrace{ \tilde \gamma_i}_{\leq \gamma} \underbrace{\|\nabla f(\bx)-\nabla f(\bar \bx_i)\|_2}_{L\|\bx-\bar \bx_i\|_2=L\gamma D_3}\underbrace{\|\bs_i-\bar \bx_i\|_2}_{\leq D_2} - \sum_i \beta_i\tilde\gamma_i \gap(\bar \bx_i) + \frac{L\gamma^2 D_2^2}{2}\\
 &\leq &  -\sum_i \beta_i\tilde\gamma_i \gap(\bar \bx_i) + \frac{L\gamma^2 D_2^2}{2} + \frac{2L\gamma^2 D_2D_3}{2}\\ 
 &\leq& -\gamma^+ \sum_i\beta_i h(\bar \bx_i)  + \frac{L\gamma^2D_2(D_2+2D_3)}{2} 
\end{eqnarray*}
where $\gamma=\gamma_k$, and $\gamma^+=\gamma_{k+1}$. Now assume $f$ is also $L_2$-continuous, e.g. $|f(\bx_1)-f(\bx_2)|\leq L_2\|\bx_1-\bx_2\|_2$. Then, taking  $h(\bx) = f(\bx) -f(\bx^*)$,

\begin{eqnarray*}
h(\bx^+)-h(\bx) 
 &\leq& -\gamma^+ \sum_i\beta_i (h(\bar \bx_i)-h(\bx)) -\gamma^+ \underbrace{\sum_i\beta_i}_{=1} h(\bx) + \frac{L\gamma^2D_2(D_2+2D_3)}{2}\\
 &\leq & 
 \gamma \sum_i\beta_iL_2 \underbrace{\|\bar \bx_i-\bx\|_2}_{\leq \gamma D_3}-\gamma^+ h(\bx) + \frac{L\gamma^2D_2(D_2+2D_3)}{2}\\
 &\leq & -\gamma^+ h(\bx) + \frac{\gamma^2(LD_2^2+2LD_2D_3+2D_3)}{2}\\
 &\leq& -\gamma^+ h(\bx) + D_4(\gamma^+)^2
\end{eqnarray*}
where 
$D_4 =\frac{LD_2^2+2LD_2D_3+2D_3}{2}$and we use $2 \geq (\gamma/\gamma^+)^2$ for all $k \geq 1$.

\end{proof}

Proof of Prop. \ref{prop:main-rungekutta-positive}
\begin{proof}
After establishing Lemma \ref{lem:rungekutta-positive-onestep}, the rest of the proof is a recursive argument, almost identical to that in  \cite{jaggi2013revisiting}. 

At $k = 0$, we define $h_0 =\max\{ h(\bx^{(0)}), \frac{ D_4c^2}{c-1}\}$, 
and it is clear that $h(\bx^{(0)}) \leq  h_0$.

Now suppose that for some $k$, $h(\bx^{(k)}) \leq \frac{h_0}{k+1}$. Then
\begin{eqnarray*}
h(x_{k+1}) &\leq &  h(\bx_k) - \gamma_{k+1}h(\bx^{(k)}) + {D_4} \gamma_{k+1}^2\\
&\leq & \frac{h_0}{k+1}\cdot \frac{k+1}{c+k+1} + D_4 \frac{c^2}{(c+k+1)^2}\\
&=& \frac{h_0}{c+k+1} + D_4 \frac{c^2}{(c+k+1)^2}\\
&=& \left( h_0 + \frac{D_4c^2}{c+k+1}\right) \left(\frac{k+2}{c+k+1}\right) \frac{1}{k+2}
\\
&\leq& h_0\left( 1+\frac{c-1}{c+k+1}\right) \left(\frac{k+2}{c+k+1}\right) \frac{1}{k+2}\\
\\
&\leq& h_0\underbrace{\left( \frac{2c+ k }{c+k+1}\right) \left(\frac{k+2}{c+k+1}\right)}_{\leq 1} \frac{1}{k+2}.
\end{eqnarray*}
\end{proof}

\subsection{Negative Runge-Kutta convergence result}
\label{app:sec:negativeresults}

This section gives the proof for Proposition \ref{prop:main-rungekutta-negative}.

\begin{lemma}[$O(1/k)$ rate]\label{lem:o1krate}
Start with $\bx_{0} = 1$. Then consider the sequence defined by
\[
\bx_{k+1}= |\bx_{k} - \frac{c_k}{k}|
\]
where, no matter how large $k$ is, there exist some constant where  $C_1 < \max_{k'>k} c_{k'} $.
(That is, although $c_k$ can be anything, the smallest upper bound of $c_k$ does not decay.) Then
\[
\sup_{k'\geq k} |\bx_{k'}| = \Omega(1/k).
\]
That is, the smallest upper bound of $|\bx_{k}|$ at least of order $1/k$.
\end{lemma}
\begin{proof}
We will show that the smallest upper bound of $|\bx_{k}|$ is larger than $C_1/(2k)$.

Proof by contradiction. 
Suppose that at some point $K$, for all $k \geq K$,  $|\bx_{k}| < C_1/(2k)$. Then from that point forward, 
\[
\sign(\bx_{k}-\frac{c_k}{k}) = -\sign(\bx_{k})
\]
and there exists some $k' > k$ where $c_{k'} > C_1$. Therefore, at that point,
\[
|\bx_{k'+1}| = \frac{c_{k'}}{k'}-|\bx_{k'}| 
\geq \frac{C_1}{2k'}>\frac{C_1}{2(k'+1)}.
\]
This immediately establishes a contradiction.
\end{proof}

Now define the operator 
\[
T(\bx_{k}) = \bx_{k+1}-\bx_{k}
\]
and note that 
\[
|\bx_{k+1}| = |\bx_{k}+T(\bx_{k})| = | |\bx_k|+\sign(\bx_{k})T(\bx_{k})|.
\]
Thus, if we can show that there exist some $\epsilon$, agnostic to $k$ (but possibly related to Runge Kutta design parameters), and
\begin{equation}
\exists k'\geq k, \quad -\sign(\bx_{k'})T(\bx_{k'}) > \frac{\epsilon}{k'},\quad \forall k,
\label{eq:lemma-helper-1}
\end{equation}
 then based on the previous lemma, this shows $\sup_{k'>k}|\bx_{k'}| = \Omega (1/k)$ as the smallest possible upper bound.

\begin{lemma}
Assuming that $0 <q\mb P^{(k)} \beta < 1$ then there exists a finite point $\tilde k$ where for all $k > \tilde k$, 
\[
|\bx_{k}| \leq \frac{C_2}{k}
\]
for some $C_2 \geq 0$.
\end{lemma}
\begin{proof}

We again use the block matrix notation
\[
\bZ^{(k)} = (\bar \bS-\bx_{k}\mb 1^T) \Gamma^{(k)}(I+A^T\Gamma^{(k)})^{-1}
\]
where $\Gamma^{(k)} = \diag(\tilde \gamma_i^{(k)})$ and each element $\tilde \gamma_i^{(k)} \leq \gamma^{(k)}$.

First, note that by construction, since 
\[
\|\bar \bS-\bx_{k}\mb 1^T\|_{2,\infty} \leq D_4, \quad \|(I+A^T\Gamma^{(k)})^{-1}\|_2 \leq  \|(I+A^T\Gamma^{(0)})^{-1}\|_2
\]
are bounded above by constants, then 
\[
\|\bZ^{(k)}\|_\infty \leq \frac{c}{c+k} C_1
\]
for $C_1 = D_4\|(I+A^T\Gamma^{(0)})^{-1}\|_2 $.

First find constants $C_3 $, $C_4$, and $\bar k$ such that
\begin{equation}
\frac{C_3}{k} \leq \mb 1^T \mathbf P^{(k)} \beta \leq \frac{C_4}{k}, \quad \forall k>\bar k,
\label{eq:boundx_helper}
\end{equation}
and such constants always exist, since
by assumption, there exists some $a_{\min} > 0$, $a_{\max}<1$ and some $k'$ where
\[
a_{\min} <q\mb P^{(k')} \beta < a_{\max} \Rightarrow \frac{a_{\min}}{q \gamma_{\max}} \leq (I+A^T\Gamma^{(k')})^{-1} \beta \leq \frac{a_{\max}}{q \gamma_{\min}}
\]
where 
\[
\gamma_{\min} = \min_i \frac{c}{c+k'+\omega^{(k')}_i}, \qquad \gamma_{\max} = \frac{c}{c+k'}.
\]
Additionally, for all $k > c+1$,
\[
\frac{c}{2k}\leq \frac{c}{c+k+1} \leq \Gamma^{(k)}_{ii} \leq \frac{c}{c+k} \leq \frac{c}{k}.
\]
Therefore taking
\[
C_3 = \frac{ca_{\min}}{2q \gamma_{\max} }, \qquad C_4 = \frac{c a_{\max}}{q\gamma_{\min}}, \qquad \bar k = \max\{k',c+1\}
\]
satisfies \eqref{eq:boundx_helper}.

Now define
\[
C_2 = \max\{|\bx_{1}|,4cq C_1 \|A\|_\infty, 4C_3, 4C_4\}.
\]
We will now inductively show that $|\bx_{k}|\leq \frac{C_2}{k}$. From the definition of $C_2$, we have the base case for $k = 1$:
\[
|\bx_{1}| \leq \frac{|\bx_{1}|}{1} \leq \frac{C_2}{k}.
\]
Now assume that $|\bx_{k}|\leq \frac{C_2}{k}$. Recall that
\[
\bx_{k+1} = \bx_{k}(1-\mb 1^T \mathbf P^{(k)} \beta) + \bar \bS\mathbf P^{(k)} \beta, \qquad \bar \bS = [\bar \bs_1,...,\bar\bs_q], \qquad \bs_i = -\sign(\bar \bx_i)
\]
and we denote the composite mixing term $\bar \gamma^{(k)} = \mb 1^T\mb P^{(k)} \beta$. 
We now look at two cases separately.
\begin{itemize}
    \item Suppose first that $\bar \bS = -\sign(\bx_{k}\mb 1^T)$, e.g. $\sign(\bar\bx_i) = \sign(\bx_{k})$ for all $i$. Then 
    \[
    \bar \bS \mb P^{(k)} \beta = -\sign(\bx_{k})\bar \gamma_k,
    \]
    and 
    \begin{eqnarray*}
    |\bx_{k+1}| &=&  |\bx_{k}(1-\bar\gamma^{(k)}) + \bar \bS\mathbf P^{(k)} \beta|\\
    &=&  |\bx_{k}(1-\bar\gamma^{(k)}) -\sign(\bx_{k})\bar\gamma^{(k)}|\\
    &=&  |\underbrace{\sign(\bx_{k})\bx_{k}}_{|\bx_{k}|}(1-\bar\gamma^{(k)}) -\underbrace{\sign(\bx_{k})\sign(\bx_{k})}_{=1}\bar\gamma^{(k)}|\\
    &=&  | |\bx_{k}|(1-\bar\gamma^{(k)}) -\bar\gamma^{(k)}|\\
    &\leq&  \max\{ |\bx_{k}|(1-\bar\gamma^{(k)}) -\bar\gamma^{(k)},
    \bar\gamma^{(k)} - |\bx_{k}|(1-\bar\gamma^{(k)}) 
    \}\\
    &\leq&  \max\Bigg\{ \underbrace{\frac{C_2}{k}(1-\frac{C_3}{k}) -\frac{C_3}{k}}_{(*)},
    \frac{C_4}{k} \Bigg\}\\
    \end{eqnarray*}
    and when $k \geq \frac{C_2}{C_3} \iff C_3 \geq \frac{C_2}{k}$,
    \[
    (*) \leq C_2\left(\frac{1}{k} -\frac{1}{k^2}\right) \leq \frac{C_2}{k+1}.
    \]
    Taking also $C_4 \leq \frac{C_2}{4}$,
    \begin{eqnarray*}
    |\bx_{k+1}| 
    \leq  \max\left\{ \frac{C_2}{k+1},
    \frac{C_2}{4k} \right\} \leq \frac{C_2}{k+1}\\
    \end{eqnarray*}
    for all $k \geq 1$.
    
    \item Now suppose that there is some $i$ where $\bar\bs_i = \sign(\bx_{k}\mb1^T)$. 
    Now since 
\[
\bar \bS = -\sign(\bx_{k}\mb 1^T + \bZ A^T)
\]
then this implies that 
$ |\bx_{k}| < (\bZ A^T)_i$.
But since 
\[
|(\bZ A^T)_i| \leq \|\bZ\|_\infty \|A\|_\infty q \leq \frac{c}{c+k}(C_1\|A\|_\infty q) \leq \frac{C_2}{4k}, 
\]
this implies that
\[
|\bx_{k+1}| \leq \frac{C_2}{4k}(1-\frac{C_3}{k}) + \frac{C_2}{4k} \leq \frac{C_2}{2k}\leq \frac{C_2}{k+1}, \quad \forall k > 1.
\]
\end{itemize}
Thus we have shown the induction step, which completes the proof.
\end{proof}

\begin{lemma}

There exists a finite point $\tilde k$ where for all $k > \tilde k$, 
\[
 \frac{c}{c+k}-\frac{C_4}{k^2} < |\xi_i| < \frac{c}{c+k}+\frac{C_4}{k^2}
\]
for some constant $C_4>0$.
\end{lemma}
\begin{proof}
Our goal is to show that 
\[
\gamma^{(k)} - \frac{C_4}{k^2} \leq \|\bZ\|_\infty \leq  \gamma^{(k)} + \frac{C_4}{k^2}
\]
for some $C_4\geq 0$, and for all $k \geq k'$ for some $k' \geq 0$.
Using the Woodbury matrix identity,
\[
\Gamma(I+A^T\Gamma)^{-1} = \Gamma\left(I - A^T(I+\Gamma A^T)^{-1} \Gamma\right)
\]
and thus
\[
\bZ^{(k)} = \bar \bS\Gamma -\underbrace{\left(\bx_{k}\mb 1^T \Gamma + (\bar \bS-\bx_{k}\mb 1^T)\Gamma A^T(I+\Gamma A^T)^{-1}\Gamma\right)}_{\mathbf B}.
\]
and thus
\[
   |\bar \bs_i \tilde \gamma_i| - \frac{C_3}{k^2} \leq |\xi_i^{(k)}| \leq |\bar \bs_i \tilde \gamma_i| + \frac{C_3}{k^2}
\]
where via triangle inequalities and norm decompositions,
\[
\frac{C_3}{k^2} = \max_i |\mathbf B_i| \leq \underbrace{|\bx_{k}|}_{O(1/k)} \gamma_k + D_4 \gamma_k^2 \|A\|_\infty (I+\Gamma^{(0)}A^T)^{-1} = O(1/k^2).
\]
Finally, since $\bar \bs_i\in \{-1,1\}$, then 
$|\bar \bs_i\tilde \gamma_i| = \tilde\gamma_i$, and in particular,
\[
\frac{c}{c+k+\omega_i}\leq \frac{c}{c+k}
\]
and
\[
\frac{c}{c+k+\omega_i}\geq \frac{c}{c+k+\omega_{\max}} = \frac{c}{c+k} - \frac{c}{c+k}\frac{\omega_{\max}}{c+k+\omega_{\max}} \geq \frac{c}{c+k}-\frac{c\omega_{\max}}{k^2}
\]
Therefore, taking $C_4 = c\omega_{\max}+C_3$ completes the proof.

\end{proof}

\begin{lemma}
There exists some large enough $\tilde k$ where for all $k \geq \tilde k$, it must be that
\begin{equation}
\exists k'\geq k, \quad -\sign(\bx_{k'})T(\bx_{k'}) > \frac{\epsilon}{k'}.
\label{eq:lemma-helper-1}
\end{equation}
\end{lemma}

\begin{proof}
Define a partitioning $S_1\cup S_2 = \{1,...,q\}$, where
\[
S_1 = \{i : \xi_i > 0\}, \quad S_2 = \{j : \xi_j\leq 0\}.
\]
Defining $\bar \xi = \frac{c}{c+k}$,
\[
|\sum_{i=1}^q \beta_i \xi_i| = |\sum_{i\in S_1}\beta_i |\xi_i| -\sum_{j\in S_2}\beta_j|\xi_j|| \geq  \left(\bar \xi-\frac{C_4}{k^2}\right)\cdot\left|\sum_{i\in S_1}\beta_i-\sum_{j\in S_2}\beta_j\right|.
\]

By assumption, there does not exist a combination of $\beta_i$ where a specific linear combination could cancel them out; that is, suppose that there exists some constant $\bar \beta$, where for \emph{every} partition of sets $S_1$,$S_2$,
\[
0<\bar\beta :=\min_{S_1,S_2}  |\sum_{i\in S_1}\beta_i-\sum_{j\in S_2}\beta_j|.
\]
Then  
\[
|\sum_{i=1}^q \beta_i \xi_i|  \geq  \left(\frac{c}{c+k}-\frac{C_2}{k^2}\right)\bar\beta \geq \bar\beta\frac{\max\{C_2,c\}}{k}.
\]
Picking $\epsilon = \max\{C_2,c\}$ concludes the proof.
\end{proof}


\section{AVERAGED FRANK WOLFE PROOFS}
\label{app:sec:averagefw}

\subsection{Accumulation terms}
\label{app:sec:accumulation}
\begin{lemma}
For an averaging term $\bar s(t)$ satisfying
\[
\dot{\bar s}(t) =  \beta(t) (s(t) - \bar s(t)), \qquad \bar s(0) = s(0) = 0
\]
where $\beta(t) = \frac{c^p}{(c+t)^p}$, then
\[
\bar s(t) = 
\begin{cases}
\displaystyle e^{-\alpha(t)}\int_0^t \frac{c^p e^{\alpha(\tau)}}{(c+\tau)^p} s(\tau) d\tau, & p\neq 1\\
\displaystyle \frac{c}{(c + t)^c} \int_0^t  (c + \tau)^{c - 1} s(\tau) d\tau & p = 1
\end{cases}
\]
where $\alpha(t) = \frac{c^p(c+t)^{1-p}}{1-p}$. If $s(t) = 1$ for all $t$, then we have an accumulation term
\[
\bar s(t) = \begin{cases}
\displaystyle  1-\frac{e^{\alpha(0)}}{e^{\alpha(t)}},& p\neq 1\\
 \displaystyle 1-(\frac{c}{c+t})^c, & p = 1\\
 \end{cases}
 \]

\end{lemma}
\begin{proof}
This can be done through simple verification.
\begin{itemize}
    \item If $p\neq 1$,
\[
\alpha'(t) = \frac{c^p}{(c+t)^p} = \beta(t), 
\]
and via chain rule,
\[
\bar s'(t) = \underbrace{e^{-\alpha(t)}\frac{c^p\exp(\alpha(t))}{(c+t)^p}}_{=\beta(t)}s(t)-\alpha'(t) \underbrace{\exp(-\alpha(t))\int_0^t \frac{c^p\exp(\alpha(\tau))}{(c+\tau)^p} s(\tau) d\tau}_{\bar s(t)}.
\]
The accumulation term can be verified if 
\[
e^{-\alpha(t)}\int_0^t \frac{c^pe^{\alpha(\tau)}}{(c+\tau)^p} d\tau = 1-\frac{e^{\alpha(0)}}{e^{\alpha(t)}}
\]
which is true since 
\[
e^{-\alpha(t)}\int_0^t \frac{c^pe^{\alpha(\tau)}}{(c+\tau)^p} d\tau =  e^{-\alpha(t)}\int_0^t(\frac{d}{d\tau} e^{\alpha(\tau)})d\tau.
\]
\item If $p = 1$
\[
\bar s'(t) = \frac{c}{(c+t)}s(t) - \frac{c^2}{(c+t)^{c+1}}\int_0^t (c+\tau)^{c-1} s(\tau) d\tau = \frac{c}{(c+t)} (s(t)-\bar s(t)).
\]
For the accumulation term, 
\[
\frac{c}{(c+t)^c}\int_0^t (c+\tau)^{c-1}  d\tau = \frac{c}{(c+t)^c}\int_0^t (\frac{\partial}{\partial \tau} \frac{(c+\tau)^c}{c})  d\tau = 1-(\frac{c}{c+t})^c.
\]
\end{itemize}

\end{proof}
For convenience, we define
\[
\beta_{t,\tau} :=
\begin{cases}
\displaystyle   \frac{c^p e^{\alpha(\tau)-\alpha(t)}}{(c+\tau)^p}   , & p\neq 1\\
\displaystyle \frac{c(c+\tau)^{c-1}}{(c+t)^b}, & p = 1,
\end{cases}
\qquad
\bar\beta_t :=
\begin{cases}
\displaystyle  1-\frac{\exp(\alpha(0))}{\exp(\alpha(t))}, & p\neq 1\\
\displaystyle 1-(\frac{c}{c+t})^b & p = 1
\end{cases}
\]

\begin{lemma}
For the averaging sequence $\bar \bs_k$ defined recursively as 
\[
\bar \bs_{k+1} = \bar \bs_k + \beta_k (\bs_k - \bar \bs_k), \qquad \bar \bs_0 = 0.
\]
Then
\[
\bar \bs_k = \sum_{i=1}^k \beta_{k,i} \bs_i, \qquad
\beta_{k,i} =  \frac{c^p}{(c+i)^p}\prod_{j=0}^{k-i-1} \left(1-\frac{c^p}{(c+k-j)^p}\right) \overset{p=1}{=} \frac{c}{c+i} \prod_{j=0}^{c} \frac{i+j+1}{c+k-j}
\]
and moreover, $\sum_{i=1}^k \beta_{k,i} = 1$.
\end{lemma}

\begin{proof}

\begin{eqnarray*}
\bar\bs_{k+1} &=& \frac{c^p}{(c+k)^p} \bs_k + \left(1-\frac{c^p}{(c+k)^p}\right)\bar\bs_k\\
&=& \frac{c^p}{(c+k)^p} \bs_k + \frac{c^p}{(c+k-1)^p}\left(1-\frac{c^p}{(c+k)^p}\right) \bs_{k-1} + \left(1-\frac{c^p}{(c+k)^p}\right)\left(1-\frac{c^p}{(c+k-1)^p}\right)\bar\bs_{k-1}\\
&=& \sum_{i=0}^k \frac{c^p}{(c+k-i)^p}\prod_{j=0}^{i-1} \left(1-\frac{c^p}{(c+k-j)^p}\right)\bs_{k-i}\\
&\overset{l = k-i}{=}& \sum_{l=1}^k \underbrace{\frac{c^p}{(c+l)^p}\prod_{j=0}^{k-l-1} \left(1-\frac{c^p}{(c+k-j)^p}\right)}_{\beta_{k,l}}\bs_{l}.
\end{eqnarray*}
If $p = 1$, then 

\[
\beta_{k,i} = \frac{c}{c+i}\prod_{l=0}^{k-i-1} \frac{k-l}{c+k-l} =\frac{c}{c+i} \frac{k(k-1)(k-2)\cdots (i+1)}{(c+k)(c+k-1)\cdots (c+i+1)}
= \frac{c}{c+i} \prod_{j=0}^{c} \frac{i+j+1}{c+k-j} 
\]

For all $p$, to show the sum is 1, we do so recursively. At $k = 1$, $\beta_{1,1} = \frac{c^p}{(c+1)^p}$. Now, if $\sum_{i=0}^{k-1} \beta_{k-1,i} = 1$, then for $i \leq k-1$
\[
\beta_{k,i} =   \left(1-\frac{c^p}{(c+k)^p}\right)\beta_{k-1,i}, \quad i \leq k-1
\]
and for $i = k$, $\beta_{k,k} = \frac{c^p}{(c+k)^p}$. Then
\[
\sum_{i=1}^k \beta_{k,i} = \beta_{k,k} +  \left(1-\frac{c^p}{(c+k)^p}\right)\sum_{l=1}^{k-1} \beta_{k-1,i} = \frac{c^p}{(c+k)^p} + \left(1-\frac{c^p}{(c+k)^p}\right) = 1.
\]
\end{proof}

\subsection{Averaging}
\label{app:sec:averaging}

In the vanilla Frank-Wolfe method, we have two players $s$ and $x$, and as $x\to x^*$, $s$ may oscillate around the solution fascet however it would like, so that its average is $x^*$ but $\|s-x^*\|_2$ remains bounded away from 0. However, we now show that if we replace $s$ with $\bar s$, whose velocity slows down, then it must be that $\|s-x^*\|_2$ decays.

\begin{lemma}[Continuous averaging]
Consider some vector trajectory $v(t)\in \mathbb R^n$, and suppose 
\begin{itemize}
    \item $\|v(t)\|_2 \leq D$ for arbitrarily large $t$

\item  $\|v'(t)\|_2 = \beta(t) D$  

\item  $\frac{1}{2}\|\int_t^\infty \gamma(\tau) v(\tau) d\tau\|_2^2 = O(1/t^q)$  for $q > 0$.
\end{itemize}
Then $\|v(t)\|^2_2 \leq O(t^{q/2+p-1})$.

\end{lemma}

\begin{proof}

We start with the orbiting property.
\[\frac{d}{dt} \left(\frac{1}{2}\|\int_t^\infty \gamma(\tau) v(\tau) d\tau\|_2^2\right) = -\int_t^\infty \gamma(t)\gamma(\tau) v(\tau)^Tv(t) d\tau \leq 0.
\]
Since this is happening asymptotically, then the negative derivative of the LHS must be upper bounded by the negative derivative of the RHS. That is, if a function is decreasing asymptotically at a certain rate, then its negative derivative should be decaying asymptotically at the negative derivative of this rate. So,

\[
\int_t^\infty \gamma(\tau) v(\tau)^Tv(t) d\tau \leq  O(1/t^q).
\]
This indicates either that $\|v(t)\|_2$ is getting smaller (converging) or $v(t)$ and its average are becoming more and more uncorrelated (orbiting).

Doing the same trick again with the negative derivative,
\[
-\frac{d}{dt}\int_t^\infty \gamma(\tau) v(\tau)^Tv(t) d\tau  = \gamma(t)\|v(t)\|_2^2 - \int_t^\infty \gamma(\tau) v(\tau)^Tv'(t)  d\tau
\]
By similar logic, this guy should also be decaying at a rate $O(1/t^{q+1})$, so 

\[
\gamma(t)\|v(t)\|_2^2  \leq \frac{C_2}{t^{q+1}} +  \int_t^\infty \gamma(\tau) v(\tau)^Tv'(t)  d\tau
\leq \frac{C_2}{t^{q+1}} +  \underbrace{\|\int_t^\infty \gamma(\tau) v(\tau)d\tau\|_2}_{\leq O(1/t^{q/2})}D\beta(t) = \frac{C_2}{t^{q+1}} + \frac{C_3}{t^{q/2+p}}
\]

Therefore 
\[
\|v(t)\|_2^2  \leq  \frac{C_2}{t^{q}} + \frac{C_3}{t^{q/2+p-1}} = O(\frac{1}{t^{q/2+p-1}}).
\]
\end{proof}

\begin{corollary}
Suppose $f$ is $\mu$-strongly convex. Then 
\[
\|\bar s(t)-x(t)\|_2^2 \leq Ct^{-(q/2+p-1)}
\]
for some constant $C>0$.
\end{corollary}
\begin{proof}
Taking $v(t) = \bar s(t)-x(t)$, it is clear that if $\beta(t) \geq \gamma(t)$ then the first two conditions are satisfied. 
In the third condition, note that 
\[
\int_t^\infty \gamma(\tau) (\bar s(\tau)-x(\tau))d\tau = \int_t^\infty \dot x(\tau) d\tau = x^*-x(t)
\]
and therefore
\[
\frac{1}{2}\|\int_t^\infty \gamma(\tau) v(\tau) d\tau\|_2^2 = \frac{1}{2}\|x^*-x(t)\|_2^2 \leq \mu (f(x)-f^*)
\]
by strong convexity.
\end{proof}

\begin{lemma}[Discrete averaging]
Consider some vector trajectory $\bv_k\in \mathbb R^n$. Then the following properties cannot all be true.
\begin{itemize}
    \item $\|\bv_k\|_2 \leq D$ for arbitrarily large $k$

\item  $\|\bv_{k+1}-\bv_k\|_2 \leq \beta_k D$  

\item  $\frac{1}{2}\|\sum_{i=k}^\infty  \gamma_k \bv_k\|_2^2 \leq \frac{C_1}{k^q}$  for $q > 0$.
\end{itemize}
Then $\|\bv_k\|^2_2 \leq O(k^{q/2+p-1})$.

\end{lemma}

\begin{proof}
The idea is to recreate the same proof steps as in the previous lemma. Note that the claimis not that these inequalities happen at each step, but that they must hold asymptotically in order for the asymptotic decay rates to hold. So
\[ 
\frac{1}{2}\|\sum_{i=k}^\infty \gamma_i \bv_i\|_2^2  -\frac{1}{2}\|\sum_{i=k}^\infty \gamma_{i+1} \bv_{i+1}\|_2^2 
=\frac{\gamma_k^2}{2}\|\bv_k\|_2^2 + \gamma_k\bv_k^T\left(\sum_{i=k}^\infty \gamma_{i+1} \bv_{i+1}\right)  \leq \frac{C_1}{(k+1)^q}-\frac{C_1}{k^q}  =\frac{C_2}{k^{q}}
\]

and therefore
\[
\frac{\gamma_k}{2}\|\bv_k\|_2^2 + \bv_k^T\left(\sum_{i=k}^\infty \gamma_{i+1} \bv_{i+1}\right) \leq \frac{C_1}{(k+1)^{q+1}}-\frac{C_1}{k^{q+1}} = \frac{C_2}{k^{q+1}}.
\]
Next,
\begin{eqnarray*}
&&\bv_k^T\left(\sum_{i=k}^\infty \gamma_{i+1} \bv_{i+1}\right)  -  \bv_{k+1}^T\left(\sum_{i=k+1}^\infty \gamma_{i+1} \bv_{i+1}\right) + \bv_{k}^T\left(\sum_{i=k+1}^\infty \gamma_{i+1} \bv_{i+1}\right)- \bv_{k}^T\left(\sum_{i=k+1}^\infty \gamma_{i+1} \bv_{i+1}\right) \\
&=&\gamma_{k+1} \bv_k^T  \bv_{k+1}  + (\bv_k-\bv_{k+1})^T\left(\sum_{i=k+1}^\infty \gamma_{i+1} \bv_{i+1}\right)
\end{eqnarray*}

\begin{eqnarray*}
\frac{\gamma_k}{2}\|\bv_k\|_2^2 - 
\frac{\gamma_{k+1}}{2}\|\bv_{k+1}\|_2^2 +\bv_k^T\left(\sum_{i=k}^\infty \gamma_{i+1} \bv_{i+1}\right)  -  \bv_{k+1}^T\left(\sum_{i=k+1}^\infty \gamma_{i+1} \bv_{i+1}\right) =\\
\frac{\gamma_k}{2}\|\bv_k\|_2^2 \underbrace{- 
\frac{\gamma_{k+1}}{2}\|\bv_{k+1}\|_2^2
+
\gamma_{k+1} \bv_k^T  \bv_{k+1}}_{-\frac{\gamma_{k+1}}{2}\|\bv_{k+1}-\bv_k\|_2^2 + \frac{\gamma_{k+1}}{2}\|\bv_k\|_2^2}\  +  (\bv_k-\bv_{k+1})^T\left(\sum_{i=k+1}^\infty \gamma_{i+1} \bv_{i+1}\right) \leq \frac{C_3}{k^{q+1}}
\end{eqnarray*}
Therefore 
\[
\frac{\gamma_k+\gamma_{k+1}}{2}\|\bv_k\|_2^2  \leq \frac{C_3}{k^{q+1}} + \underbrace{(\bv_{k+1}-\bv_k)^T\left(\sum_{i=k+1}^\infty \gamma_{i+1} \bv_{i+1}\right)}_{O(\beta_k/k^{q/2})} + \underbrace{\frac{\gamma_{k+1}}{2}\|\bv_{k+1}-\bv_k\|_2^2}_{O(\gamma_k\beta_k^2)}
\]
Finally,
\[
\|\bv_k\|_2^2  \leq  \frac{C_3}{k^{q}} + \frac{C_4}{k^{q/2+p-1}} + \frac{C_5}{k^{2p}} = O(1/k^{\min\{q/2+p-1,2p\}}).
\]
\end{proof}

\begin{corollary}
Suppose $f$ is $\mu$-strongly convex. Then if $f(x)-f^* = O(k^{-q})$
\[
\|\bar \bs_k-\bx_k\|_2^2 \leq C\max\{k^{-(q/2+p-1)},k^{-2p}\} 
\]
for some constant $C>0$.
\end{corollary}
\begin{proof}
Taking $\bv_k = \bar \bs_k-\bx_k$, it is clear that if $\beta(t) \geq \gamma(t)$ then the first two conditions are satisfied. 
In the third condition, note that 
\[
\sum_{i=k}^\infty \gamma_i (\bar \bs_i-\bx_i) = \sum_{i=k}^\infty \bx_{i+1}-\bx_i = \bx^*-\bx_k
\]
and therefore
\[
\frac{1}{2}\|\sum_{i=k}^\infty \gamma_i (\bar \bs_i-\bx_i)\|_2^2 = \frac{1}{2}\|\bx^*-\bx_k\|_2^2 \leq \mu (f(\bx_k)-f^*)
\]
by strong convexity.
\end{proof}

\subsection{Global rates}
\label{app:sec:global}

\begin{lemma}[Continuous energy function decay]
\label{lem:continuous_energy}
Suppose $c \geq q-1$, and
\[
 g(t) \leq   -\int_0^t \frac{\exp(\alpha(\tau))}{\exp(\alpha(t))}\frac{(c+\tau)^c}{(c + t)^{c}} \frac{C_1}{(b+\tau)^r}d\tau
\]
Then
\[
g(t) \leq -\frac{\left(1-\frac{\alpha(1)}{\exp(\alpha(1))}\right)C_1}{\alpha(t)(1-p)}(c+t)^{1-r} 
\]
 
\end{lemma}
\begin{proof}

\begin{eqnarray*}
\frac{g(t)}{C_1}\exp(\alpha(t))(c+t)^{c}&\leq&- \int_0^t \exp(\alpha(\tau)) (c+\tau)^{c-r} d\tau\\
&=& - \int_0^t \sum_{k=1}^\infty \frac{\alpha(\tau)^k}{k!}(c+\tau)^{c-r} d\tau\\
&=& -\int_0^t \sum_{k=1}^\infty \frac{c^{kp}}{k!(1-p)^k}(c+\tau)^{k-pk+c-r} d\tau\\
&\overset{\text{Fubini}}{=}&  -\sum_{k=1}^\infty \frac{c^{kp}}{(1-p)^kk!}\int_0^t (c+\tau)^{k-pk+c-r} d\tau\\
&=&  - \sum_{k=1}^\infty \frac{c^{kp}}{(1-p)^kk!} \frac{(c+t)^{k-pk+c-r+1}-c^{k-pk+c-r+1}}{k-pk+c-r+1} \\
&=&   -\sum_{k=1}^\infty \frac{1}{(k+1)!}\left(\frac{c^{p}}{(c+t)^{p-1}(1-p)}\right)^k(c+t)^{1+c-r} \underbrace{\frac{1}{(1-p)+(c-r+1)/k}\frac{k+1}{k}}_{\geq C_2} \\
&\leq& -C_2 \sum_{k=1}^\infty \frac{(c+t)^{1+c-r}}{(k+1)!}\frac{\alpha(t)^{k+1}}{\alpha(t)}\\
&=& -\frac{C_2(c+t)^{1+c-r} }{\alpha(t)}(\exp(\alpha(t))-\alpha(1))
\end{eqnarray*}
Then

\begin{eqnarray*}
g(t) &\leq& -\frac{C_1C_2}{\alpha(t)}(c+t)^{1-r}\left(1-\frac{\alpha(1)}{\exp(\alpha(t))}\right)\\
 &\leq& -\frac{C_1C_2}{\alpha(t)}(c+t)^{1-r}\left(1-\frac{\alpha(1)}{\exp(\alpha(1))}\right)\\
 &\leq& -\frac{C_1C_3}{\alpha(t)}(c+t)^{1-r}
\end{eqnarray*}
where $C_3 = C_2\left(1-\frac{\alpha(1)}{\exp(\alpha(1))}\right)$  and $C_2 = \frac{1}{1-p}$ satisfies the condition.
\end{proof}
\begin{theorem}[Continuous global rate]
Suppose $0 < p < 1$. 
Then the averaged FW flow decays as $O(1/t^{1-p})$.
\end{theorem}

\begin{proof}
Consider the error function 
\[
g(t) := \nabla f(x(t))^T(\bar s(t)-x(t)), \qquad g(0) = 0.
\]
\begin{eqnarray*}
\dot g(t) &=& \frac{\partial}{\partial t} \nabla f(x)^T(\bar s-x)\\
&=& \left(\frac{\partial}{\partial t}\nabla f(x)^T \right)(\bar s - x)
+
\nabla f(x)^T\left(\frac{\partial}{\partial t}(\bar s - x)\right)
 \\
&=&  \underbrace{ \left(\frac{\partial}{\partial t} \nabla f(x)^T \right)}_{={\dot x}^T\nabla^2 f(x)}(\bar s - x)
+
\nabla f(x)^T \left(\beta(t) (s(t)-\bar s(t)) - \gamma(t)(\bar s(t)-x(t))\right)\\
&\leq & \gamma \underbrace{(\bar s - x)^T\nabla^2 f(x) (\bar s - x)}_{\leq 4LD^2\gamma(t)} + \underbrace{\beta(t)\nabla f(x)^T(s(t)-x(t))}_{-\beta(t)\gap(x)} - (\beta(t)+\gamma(t)) \underbrace{\nabla f(x)^T (\bar s(t) -x(t))}_{=g(t)} 
\end{eqnarray*}

\begin{eqnarray*}
g(t) &\leq&    \int_0^t \frac{\exp(\alpha(\tau))}{\exp(\alpha(t))}\frac{(c+\tau)^c}{(c + t)^{c}} \underbrace{\left(\frac{4LD^2c}{c+\tau} -  \frac{b^p}{(b+\tau)^p}\gap(x(\tau))\right)}_{A(\tau)}d\tau\\
\dot h(t) &\leq&  \gamma(t)g(t) \leq  \gamma(t) \int_0^t \underbrace{\frac{\exp(\alpha(\tau))}{\exp(\alpha(t))}\frac{(c+\tau)^c}{(c + t)^{c}}}_{\mu(\tau)} A(\tau) d\tau 
\end{eqnarray*}
In order for $h(t)$ to decrease, it must be that $\dot h(t) \leq 0$. However, since $\mu(\tau) \geq 0$ for all $\tau \geq 0$, it must be that $A(\tau) \leq 0$, e.g.
\[
  \frac{c^p}{(c+\tau)^p}\gap(x(\tau))\geq \frac{4LD^2c}{c+\tau}.
\]
which would imply $h(t) = O(1/(c+t)^{1-p})$. Let us therefore test the candidate solution
\[
h(t) = \frac{C_3}{(c+t)^{1-p}}.
\]
Additionally, from Lemma \ref{lem:continuous_energy}, if 
\[
A(\tau) \leq -\frac{C_1}{(c+t)}
\quad \Rightarrow\quad
g(t) \leq -\frac{C_1}{\alpha(t)(1-p)}
\]
and therefore 
\begin{eqnarray*}
\dot h(t) &\leq&  \gamma(t)g(t)  \leq -\frac{c}{c+t} \frac{C_1}{c^{p}}\cdot(c+t)^{p-1} \\
\int_0^t \dot(h(\tau))d\tau&\leq& \frac{C_1 c}{(1-p)c^p}(c+t)^{p-1}
\end{eqnarray*}
which satisfies our candidate solution for $C_3 = \frac{C_1 c}{(p-1)c^p}$.
\end{proof}

This term $ (\bar s - x)^T\nabla^2 f(x) (\bar s - x)\leq 4LD^2\gamma(t)$ is an important one to consider when talking about local vs global distance. The largest values of the Hessian will probably not correspond to the indices that are ``active'', and thus this bound is very loose near optimality.

\begin{lemma}[Discrete energy decay]
\label{lem:discrete_energy}
Suppose $0 < p < 1$. 
Consider the error function 
\[
\bg_k := \nabla f(\bx_k)^T(\bar\bs_k-\bx_t).
\]

Then
\[
\bg_k \leq -\sum_{i=0}^{k-1} \beta_{i,i} (\frac{i+1+c}{k+c})^c \gap(\bx_i)   - \beta_{k,k}\gap(\bx_k)+  \frac{4D^2L C_1}{(k+c)^p} + (\frac{c}{k+c})^c \bg_0.
\]
where $C_1 = c^p(1+\frac{1}{(c+1)(1-p)-1})$.
\end{lemma}
Importantly, $C_1$ is finite only if $p < 1$. When $p = 1$, the right hand side is at best bounded by a constant, and does not decay, which makes it impossible to show method convergence.

\begin{proof}
Define $\bz_k = \nabla f(\bx_k)$, $\bg_k = \bz_{k}^T(\bar\bs_{k}-\bx_{k})$.  Then 

\begin{eqnarray*}
\bg_k &=&\underbrace{\beta_{k,k} \bz_k^T(\bs_k-\bx_k) }_{-\beta_{k,k}\gap(\bx_k)} 
+ \underbrace{\sum_{i=0}^{k-1}\beta_{k,i}\bz_k^T(\bs_i-\bx_k)}_{A}\\
A&=& \sum_{i=0}^{k-1}\underbrace{\frac{\beta_{k,i}}{\beta_{k-1,i}}}_{=(1-\frac{c}{(c+i)^p})}\beta_{k-1,i} \bz_k^T(\bs_i-\bx_k)\leq  (1-\frac{c^p}{(c+k)^p}) \underbrace{\bz_k^T(\bar\bs_{k-1}-\bx_k)}_{=B}  \\
B&=&  \bz_k^T(\bar\bs_{k-1}-\underbrace{(\bx_{k-1}+\gamma_{k-1}(\bar\bs_{k-1}-\bx_{k-1}))}_{\bx_{k}}) \\
&=& (1-\gamma_{k-1})\bz_{k}^T(\bar\bs_{k-1}-\bx_{k-1})\\
&=& (1-\gamma_{k-1})(\bz_{k}-\bz_{k-1})^T\underbrace{(\bar\bs_{k-1}-\bx_{k-1})}_{(\bx_k-\bx_{k-1})\gamma_{k-1}^{-1}}
+ (1-\gamma_{k-1})\underbrace{\bz_{k-1}^T(\bar\bs_{k-1}-\bx_{k-1})}_{\bg_{k-1}} \\
&=& \frac{(1-\gamma_{k-1})}{\gamma_{k-1}} \underbrace{(\bz_{k}-\bz_{k-1})^T (\bx_k-\bx_{k-1})}_{\leq L\|\bx_k-\bx_{k-1}\|_2^2}
+ (1-\gamma_{k-1}) \bg_{k-1} \\
&\leq& (1-\gamma_{k-1})\gamma_{k-1}L \underbrace{\|\bar \bs_{k-1}-\bx_{k-1}\|_2^2}_{4D^2}
+ (1-\gamma_{k-1}) \bg_{k-1} \\
\end{eqnarray*}
Overall,
\begin{eqnarray*}
\bg_k &\leq& -\beta_{k,k}\gap(\bx_k) + 4D^2L\gamma_{k-1}(1-\gamma_{k-1})(1-\beta_k)+\underbrace{(1-\beta_k)(1-\gamma_{k-1})}_{\mu_k}\bg_{k-1}\\
&=& -\beta_{k,k}\gap(\bx_k) + 4D^2L \gamma_{k-1}\mu_k + \mu_k\bg_{k-1}\\
&=& -\beta_{k,k}\gap(\bx_k) + 4D^2L  \gamma_{k-1} \mu_k
-\beta_{k-1,k-1}\mu_k\gap(\bx_{k-1}) + 4D^2L \gamma_{k-2}\mu_{k-1}\mu_k + \mu_k\mu_{k-1}\bg_{k-2} 
\\
&=& -\sum_{i=0}^{k-1}\beta_{i,i}\gap(\bx_i)\prod_{j=i+1}^{k}\mu_j - \beta_{k,k}\gap(\bx_k) + 4D^2L\sum_{i=0}^k\gamma_{k-i}\prod_{j=i}^k\mu_j+ \prod_{j=1}^k\mu_k \bg_0
\end{eqnarray*}

Now we compute $\prod_{j=i}^k \mu_j$
\[
\prod_{j=i}^k(1-\gamma_{j-1}) = \prod_{j=i}^k\frac{j-1}{c+j-1} = \prod_{j=0}^c \frac{i-1+j}{k+j} \leq  (\frac{i+c}{k+c})^c
\]
Using $1-\frac{c^p}{(c+k)^p}\leq \exp(-(\frac{c}{c+k})^p)$, 
\begin{eqnarray*}
\log(\prod_{j=i}^k(1-\frac{c^p}{(c+j)^p})) \leq -\sum_{j=i}^k (\frac{c}{c+j})^p \leq  -\int_{i}^k (\frac{c}{c+j})^p dj 
=  \frac{c^p}{p+1} ((c+i)^{p+1}-(c+k)^{p+1})
\end{eqnarray*}
and therefore
\[
\prod_{j=i}^k(1-\frac{c^p}{(c+k)^p}) \leq  \frac{\exp(\frac{c^p}{p+1} (c+i)^{p+1})}{\exp(\frac{c^p}{p+1}(c+k)^{p+1})}
\]
which means
\[
\prod_{j=i}^k\mu_j \leq (\frac{i+c}{k+c})^c \exp(\frac{c^p}{p+1} ((c+i)^{p+1}-(c+k)^{p+1})).
\]
Now we bound the constant term coefficient.
\begin{eqnarray*}
\sum_{i=0}^k \gamma_{i-1}\prod_{j=i}^k\mu_j &\leq&
\sum_{i=0}^k\underbrace{\frac{c}{(c+i-1)}  (\frac{i+c}{k+c})^c}_{\text{max at $i=0$}} \underbrace{\exp(\frac{c^p}{p+1} ((c+i)^{p+1}-(c+k)^{p+1}))}_{\leq \frac{1}{\left(k-i\right)^{\left(c+1\right)\left(1-p\right)}}}\\
&\overset{C-S}{\leq}&\frac{c}{c-1}\frac{c^c}{(k+c)^c}  \sum_{i=0}^{k-1}  \frac{1}{\left(k-i\right)^{(c+1)(1-p)}}\\
&\leq&\frac{c}{c-1}\frac{c^c}{(k+c)^c}  \frac{1}{(c+1)(1-p)-1} \frac{1}{(1 - k)^{ (c+1)(1-p)-1}}\\
&\leq& \frac{C_1}{(k+c)^p}\frac{1}{(1 - k)^{ (c+1)(1-p)-1}}
\end{eqnarray*}
where $(*)$ if $c$ is chosen such that $(c+1)(1-p) > 1$ and $C_1>0$ big enough. Note that necessarily, $p < 1$, and the size of $C_1$ depends on how close $p$ is to 1.

Also, to simplify terms,
\[
\prod_{j=i}^k\mu_j \leq (\frac{i+c}{k+c})^c \underbrace{\exp(\frac{c^p}{p+1} ((c+i)^{p+1}-(c+k)^{p+1})).}_{\leq 1}
\]

Now, we can say
\[
\bg_k \leq -\sum_{i=0}^{k-1} \beta_{i,i} (\frac{i+1+c}{k+c})^c \gap(\bx_i)   - \beta_{k,k}\gap(\bx_k)+  \frac{4D^2L C_1}{(k+c)^p} + (\frac{c}{k+c})^c \bg_0.
\]

\end{proof}

\begin{theorem}[Global rate, $p < 1$.]
\label{th:globalrate}
Suppose $0 < p < 1$ and 
$c \geq \frac{c-1}{c^p}$. Then
 $h(\bx_k) =: \bh_k = O(\tfrac{1}{(k+c)^p})$.
\end{theorem}

\begin{proof}Start with
\[
\bg_k \leq -\sum_{i=0}^{k-1} \beta_{i,i} (\frac{i+1+c}{k+c})^c \gap(\bx_i)   - \beta_{k,k}\gap(\bx_k)+  \frac{4D^2L C_1}{(k+c)^p} + (\frac{c}{k+c})^c \bg_0.
\]
\begin{eqnarray*}
\bh_{k+1}-\bh_k &\leq& \gamma_k\bg_k + 2\gamma_k LD^2\\
&\leq & -\frac{c}{c+k}\frac{1}{(k+c)^c}\sum_{i=0}^{k-1} \beta_{i,i} (i+c+1)^c  \bh_i - \frac{c}{c+k}\beta_{k,k}\bh_k\\
 &&\qquad + 2D^2L\gamma_k( \frac{2C_1}{(k+c)^p} + \gamma_k) + \frac{c}{c+k}\frac{c^c}{(k+c)^c} \bg_0 \\
 &\leq&  -\frac{c^{p+1}}{c+k}\frac{1}{(k+c)^c}\sum_{i=0}^{k-1} (i+c)^{c-p}  \bh_i - \frac{c}{c+k}\beta_{k,k}\bh_k\\
 &&\qquad + 2D^2L\gamma_k( \frac{2C_1}{(k+c)^p} + \gamma_k) + \frac{c}{c+k}\frac{c^c}{(k+c)^c} \bg_0 \\
\end{eqnarray*}

Suppose $\bh_{k} \leq \frac{C_2}{(k+c)^p}$. Then
\begin{eqnarray*}
\bh_{k+1}  -\bh_k &\leq&   \underbrace{-\frac{c^{p+1}}{c+k}\frac{C_2}{(k+c)^c}\sum_{i=0}^{k-1} (c+i)^{c-2p}    - \frac{cC_2}{c+k}\frac{c^p}{(c+k)^{2p}}}_{-AC_2}\\
 &&\qquad + \underbrace{2D^2L\frac{c}{c+k}\left( \frac{2C_1}{(k+c)^p} + \frac{c}{c+k}\right) + \frac{c}{c+k}\frac{c^c}{(k+c)^c} \bg_0}_{B} \\
 B&=& \frac{c}{c+k}\left(2D^2L\left( \frac{2C_1}{(k+c)^p} + \frac{c}{c+k}\right) + \frac{c^c}{(k+c)^c} \bg_0\right)
 \\
 &\leq &\frac{2cD^2L( 2C_1+c ) + c^{c+1} \bg_0}{(c+k)^{p}(c+k)}\\
 &=:&\frac{C_3}{(c+k)^{p+1}}\\
\end{eqnarray*}
where $c > 1$.  Then,

\begin{eqnarray*}
 (k+c)^cA &=&  \frac{c^{p+1}}{c+k}\sum_{i=0}^{k-1} (c+i)^{c-2p}    + \frac{c}{c+k}\frac{c^p}{(c+k)^{2p-1}}\\
 &\geq & \frac{c^{p+1}}{c+k} \frac{(c+k-1)^{c-2p+1}-(c+1)^{c-2p+1}}{c-2p+1} + \frac{c}{c+k} \frac{c^p}{(c+k)^{2p-1}}\\
 &= & \frac{c^{p+1}}{c-2p+1}\frac{(c+k-1)^{c-2p+1}}{c+k} + O(1/k)\\ 
 &\overset{c\geq 2p+1}{\geq} & \frac{c^{p+1}}{c-2p+1} + O(1/k)\\
 &\overset{\text{$k$ big enough}}{\geq}& \frac{c^{p+1}}{c-1}
 \end{eqnarray*}

Therefore,
\begin{eqnarray*}
\bh_{k+1}    &\leq&  \frac{C_2(1-\frac{c^{p+1}}{c-1})}{(k+c)^p} + \frac{C_3}{(c+k)^{p+1}}.
\end{eqnarray*}
Define 
$\epsilon = \frac{2c^{p+1}}{c^{p+1} + 1 - c}$
and pick $C_2 > \frac{C_3}{c\epsilon}$. By assumption, $\epsilon > 0$. Consider $k > K$ such that for all $k$, 
\[
\frac{(k+c+1)^p}{(k+c)^p} \leq 1+\frac{\epsilon}{2}, \qquad \frac{C_3}{c+k} \leq \frac{\epsilon}{2}\frac{(c+k)^p}{(c+k+1)^p}. 
\]
Then
\begin{eqnarray*}
\bh_{k+1}    &\leq&  \frac{C_2(1-\tfrac{\epsilon}{2})}{(k+c+1)^p} + \frac{\tfrac{\epsilon}{2}}{(k+c+1)^p} \leq \frac{C_2}{(k+c+1)^p}.
\end{eqnarray*}
We have now proved the inductive step.
Picking $C_2 \geq \bh_0$ gives the starting condition, completing the proof.

\end{proof}

\subsection{Local rates}
\label{app:sec:local}
\begin{lemma}[Local convergence]
Define, for all  $t$,
\begin{equation}
\tilde s(t) = \argmin{\tilde s\in \conv(\mS(\bx^*))}\|s(t)-\tilde s\|_2, \qquad \hat s(t) = \bar \beta_t^{-1}\int_{0}^t \beta_{t,\tau} \tilde s(\tau) d\tau.
\label{eq:def-shat}
\end{equation}
e.g., $\hat s(t)$ is the closest point in the convex hull of the support of $\bx^*$ to the the point $s(t) = \lmo_\mD(x(t))$.

Then, 
\[
\|\bar s(t) - \hat s(t)\|_2 \leq \frac{c_1}{(c+t)^c}.
\]
\end{lemma}

The proof of this theorem actually does not really depend on how well the FW method works inherently, but rather is a consequence of the averaging. Intuitively, the proof states that after the manifold has been identified, all new support components must also be in the convex hull of the support set of the \emph{optimal} solution; thus, in fact $s_2(t) - \hat s(t) = 0$. However, because the accumulation term in the flow actually is not a true average until $t\to +\infty$, there is a pesky normalization term which must be accounted for. Note, importantly, that this normalization term does not appear in the method, where the accumulation weights always equal 1 (pure averaging).

\begin{proof}
First, note that 
\[
\bar s(t) - \hat s(t) =\bar s(t) - \bar \beta_t^{-1} \int_0^t \beta_{t,\tau} \tilde s(\tau) d\tau = \int_0^t \beta_{t,\tau} (s(\tau) - \bar\beta_t^{-1} \tilde s(\tau)) d\tau.
\]
Using triangle inequality,
\[
\|\bar s(t) - \hat s(t)\|_2 \leq \underbrace{ \|\int_0^{\bar t} \beta_{t,\tau} (s(\tau) - \bar\beta_t^{-1} \tilde s(\tau)) d\tau\|_2}_{\epsilon_1} + \underbrace{\|\int_{\bar t}^t \beta_{t,\tau} (s(\tau) - \bar\beta_t^{-1} \tilde s(\tau)) d\tau\|_2}_{\epsilon_2}.
\]

Expanding the first term, 
via Cauchy Scwhartz for integrals, we can write, elementwise,
\[
\int_0^{\bar t} \beta_{t,\tau} (s(\tau)_i - \bar\beta_t^{-1} \tilde s(\tau)_i) d\tau   \leq \int_0^{\bar t} \beta_{t,\tau}  d\tau \int_0^{\bar t} |s(\tau)_i - \bar\beta_{ t}^{-1} \tilde s(\tau)_i| d\tau
\]
and thus, 
\[
\|\int_0^{ t} \beta_{t,\tau} (s(\tau) - \bar\beta_t^{-1} \tilde s(\tau)) d\tau\|   \leq \int_0^{\bar t} \beta_{t,\tau}   d\tau \underbrace{\|\int_0^{ t} s(\tau)_i - \bar\beta_t^{-1} \tilde s(\tau)_i d\tau\|_2}_{\leq 2 D (1+\bar \beta_{t}^{-1}) \bar t} .
\]
and moreover,
\[
\int_0^{\bar t} \beta_{t,\tau} d\tau = \int_0^{\bar t} \frac{c(c+\tau)^{c-1}}{(c+t)^c} d\tau = \frac{\hat c_0}{(c+t)^c}
\]
since $\int_0^{\bar t} b(b+\tau)^{b-1}  d\tau$ does not depend on $t$.
Thus the first error term 
\[
\epsilon_1 \leq \frac{2\hat c_0 D \bar t (1+\bar \beta_t^{-1}) }{(c+t)^c} \leq \frac{c_0}{(c+t)^c}
\]
where
\[
\hat c_0 :=  2D\bar t \int_0^{\bar t} c(c+\tau)^{c-1} d\tau.
\]

In the second error term, \emph{because the manifold has now been identified}, $s(\tau)  = \tilde s(\tau)$, and so 
\[
\int_{\bar t}^t \beta_{t,\tau} (s(\tau) - \bar\beta_t^{-1} \tilde s(\tau)) d\tau   
=
\int_{\bar t}^t \beta_{t,\tau}(1-\bar \beta_t^{-1}) s(\tau)  d\tau   
\]
and using the same Cauchy-Schwartz argument, 
\[
\int_{\bar t}^t \beta_{t,\tau}(1-\bar \beta_t^{-1}) s(\tau)  d\tau \leq D\int_{\bar t}^t \beta_{t,\tau}(1-\bar \beta_t^{-1})    d\tau.
\]
The term
\[
1-\bar \beta_t^{-1} = |1-\frac{1}{1-(\frac{c}{c+t})^c}| =\frac{c^c}{(c+t)^c-c^c}\leq \frac{2c^2 }{(c+t)^{c}}
\] 
and thus
\[
\epsilon_t \leq \int_{\bar t}^t \beta_{t,\tau}(1-\bar\beta_t^{-1}) d\tau \leq \frac{2 c^3}{(c+t)^{2c}}\int_{\bar t}^t (c+\tau)^{c-1}  = 
\frac{2c^2 }{(c+t)^{2c}}( (c+t)^c - (c+\bar t)^c)  \leq \frac{2c^2}{(c+t)^c}.
\]
Thus,
\[
\|\bar s(t)-\hat s(t)\|_2  \leq \frac{\hat c_0+2c^2}{(c+t)^c} = O(\frac{1}{(c+t)^c}).
\]

\end{proof}

\begin{corollary}[Local flow rate]
Suppose that for all $x\in \mD$, $\|\nabla f(x)\|_2 \leq G$ for some $G$ large enough. 
Consider $\gamma(t) = \beta(t) = \frac{c}{c+t}$. Then the ODE
\[
\dot h(x(t)) = \gamma(t)\nabla f(x)^T(\bar s - x)
\]
has solutions 
  $h(t) = O(\frac{\log(t)}{(c+t)^c})$
   when $t \geq \bar t$.

\end{corollary}
\begin{proof}
First, we rewrite the ODE in a more familiar way, with an extra error term
\[
\dot h(x(t)) = \gamma(t)\nabla f(x)^T(\bar s - \hat s) +  \gamma(t)\nabla f(x)^T(\hat s - x)
\]
where $\hat s$ is as defined in \eqref{eq:def-shat}. By construction, $\hat s$ is a convex combination of $\tilde s\in \mS(\bx^*)$. Moreover, after $t \geq \bar t$, $\mS(\bar x(t)) = \mS(\bx^*)$, and thus 
\[
\nabla f(x)^T(\hat s(t) - x) = \nabla f(x)^T(s(t) - x) = -\gap(t) \leq -h(t).
\]
Then, using Cauchy-Schwartz, and piecing it together,
\[
h(t) = \nabla f(x)^T(\hat s(t) - x) \leq G\gamma(t)\|\bar s - \hat s\|_2 - \gamma(t) h(t) \leq \frac{G\gamma(t)  c_1}{(c+t)^c} - \gamma(t) h(t).
\]

 Let us therefore consider the system
  \[
  \dot h(x(t)) = \frac{2 G D \gamma(t) }{(c+t)^c}   - \gamma(t) h(x(t)) .
\]
The solution to this ODE is
 \[
 h(t) = \frac{h(0)c^c + 2GDc \log(c+t) - 2GDc\log(c)}{(c + t)^c} = O(\frac{\log(t)}{(c+t)^c}).
 \]

\end{proof}

\begin{lemma}[Local  averaging error]
\label{lem:local_avgerror}
Define, for all  $k$,
\[
\tilde \bs_k = \argmin{\tilde \bs\in \conv(\mS(\bx^*))}\|\bs_k-\tilde \bs\|_2, \qquad \hat \bs_k = \bar \beta_k^{-1}\sum_{i=1}^k \beta_{k,i} \tilde \bs_i.
\]
e.g., $\tilde \bs(k)$ is the closest point in the convex hull of the support of $\bx^*$ to the the point $\bs_k = \lmo_\mD(\bx_k)$.

Then, 
\[
\|\bar \bs_k - \hat \bs_k\|_2 \leq \frac{c_2}{k^c}.
\]
\end{lemma}

\begin{proof}
First, note that 
\[
\bar \bs_k - \hat \bs_k =\bar \bs_k -  \sum_{i=1}^k \beta_{k,i} \tilde \bs_i = \sum_{i=1}^k  \beta_{k,i} (\bs_i -  \tilde \bs_i).
\]
Using triangle inequality,
\[
\|\bar \bs_k - \hat \bs_k\|_2 \leq \underbrace{ \|\sum_{i=1}^{\bar k} \beta_{k,i} (\bs_i -  \tilde \bs_i)) \|_2}_{\epsilon} + \underbrace{\|\sum_{i=\bar k}^k \beta_{k,i} (\bs_i -  \tilde \bs_i) \|_2}_{0}.
\]
where the second error term is 0 since the manifold has been identified, so $\tilde \bs_i = \bs_i$ for all $i \geq \bar k$.

Expanding the first term,  using a Holder norm (1 and $\infty$ norm) argument,
\begin{eqnarray*}
\|\sum_{i=1}^{\bar k} \beta_{k,i} (\bs_i -  \tilde \bs_i)  \|_2     &\leq& 2D\sum_{i=1}^{\bar k} \beta_{k,i} \\
&=&  2D\sum_{i=1}^{\bar k} \frac{c}{(c+i)}\prod_{j=0}^{ k - i - 1}(1-\frac{c}{(c+k-j)}) \\
&=& 2D\sum_{i=1}^{\bar k} \frac{c}{(c+i)} \prod_{j=0}^c \frac{i-1+j}{c+k-j}\\
&\leq& 2D (\frac{\bar k - 1 + c}{k})^c\sum_{i=1}^{\bar k} \frac{c}{(c+i)}= O(1/k^c).
\end{eqnarray*}
\end{proof}

\begin{corollary}[Local convergence rate bounds]
\label{lem:local_convratebnd}
Suppose that for all $\bx\in \mD$, $\|\nabla f(\bx)\|_2 \leq G$ for some $G$ large enough. Define also $r$ the decay constant of $\|\bar\bs_k-\bx_k\|^2_2$ ($=O(1/k^r)$).   
Consider $\gamma_k = \beta_k = \frac{c}{c+k}$. Then the difference equation
\[
\bh(\bx_{k+1})-\bh(\bx_k) \leq \gamma_k\nabla f(\bx)^T(\bar \bs_k - \bx_k) + \frac{C}{k^r}
\]
is satisfied with candidate solution 
 $\bh(\bx_k) = C_4 \max\{\frac{\log(k)}{(c+k)^c},\frac{1}{k^{r+1}}\}$
   when $k \geq \bar k$.

\end{corollary}

\begin{proof}
First, we rewrite the ODE in a more familiar way, with an extra error term
\[
\bh(\bx_{k+1})-\bh(\bx_k) = \gamma_k\underbrace{\nabla f(\bx_k)^T(\bar \bs_k - \hat \bs_k)}_{\leq \gamma_k G\|\bar\bs_k-\hat\bs_k\|_2} +  \gamma_k\nabla f(\bx_k)^T(\hat \bs_k - \bx_k) + \frac{C}{k^{r+2}}
\]
where $\hat \bs_k$ is as defined in \eqref{eq:def-shat}. By construction, $\hat \bs_k$ is a convex combination of $\tilde \bs_i\in \mS(\bx^*)$. Moreover, after $k \geq \bar k$, $\mS(\bar \bx_k) = \mS(\bx^*)$, and thus 
\[
\nabla f(\bx_k)^T(\hat \bs_k - \bx_k) = \nabla f(\bx_k)^T(\bs_k - \bx_k) = -\gap(\bx_k) \leq -\bh(\bx_k).
\]
Then,   piecing it together,
\[
\bh(\bx_{k+1})-\bh(\bx_k) \leq \gamma_kG\|\bar\bs-\hat\bs\|_2 -\gamma_k\bh(\bx_k) + \frac{C}{k^{r+2}} \overset{\text{Lemma \ref{lem:local_avgerror}}}{\leq } \underbrace{\gamma_k\frac{GC_2}{k^c}}_{\leq C_3/k^{c+1}} -\gamma_k\bh(\bx_k) + \frac{C}{k^{r+2}}
\]

Recursively, we can now show that for $C_4 \geq C+C_3$, if
\[
\bh(\bx_k) \leq C_4 \max\{\frac{\log(k)}{(c+k)^c},\frac{1}{k^{r+1}}\}
\]
then, 
\begin{eqnarray*}
\bh(\bx_{k+1})&\leq& \frac{C_3}{k^{c+1}}+ \frac{C}{k^{r+1}}+(1-\gamma_k) \bh(\bx_k) \\
&\leq & \frac{C_3}{k^{c+1}}+ \frac{C}{k^{r+1}}+\frac{k}{c+k}  C_4\max\{\frac{\log(k)}{(c+k)^c},\frac{1}{k^{r+1}}\}.
\end{eqnarray*}
If $c\leq r+1$ then 

\begin{eqnarray*}
\bh(\bx_{k+1})&\leq&  \frac{C+C_3}{k^c}+\frac{k}{c+k}  \frac{C_4\log(k)}{(c+k)^c}\leq \frac{k}{c+k}  \frac{C_4\log(k)}{(c+k)^c}\leq   \frac{C_4\log(k+1)}{(c+k+1)^c}
\end{eqnarray*}
for $k$ large enough. Otherwise,

\begin{eqnarray*}
\bh(\bx_{k+1})&\leq&  \frac{C+C_3}{k^{r+2}}+\frac{k}{c+k}  \frac{C_4}{k^{r+1}}\leq \frac{C_4}{(k+1)^{r+1}}.
\end{eqnarray*}
for $k$ large enough.
\end{proof}

\begin{theorem}[Local convergence rate]
Picking $c \geq 3p/2+1$, the proposed method \AvgFW~has an overall convergence $\bh(\bx_k) = O(k^{-3p/2})$.
\end{theorem}

\begin{proof}
Putting together Theorem \ref{th:globalrate}, Lemma \ref{lem:local_avgerror}, and Corollary \ref{lem:local_convratebnd}, we can resolve the constants
\[
q = p, \qquad r = \min\{q/2+p-1,2p\} = \frac{3p}{2}-1
\]
and  resolves an overall convergence bound of 
$\bh(\bx_k) = O(k^{-3p/2})$.

\end{proof}

\end{document}